\documentclass[a4paper,12pt, reqno]{amsart}
\usepackage{amssymb,amsthm,amsmath}
\usepackage{cite}
\usepackage{mathabx}
\usepackage{mathtools}
\usepackage{amsmath,amsthm,amscd,amssymb}
\usepackage{latexsym}
\usepackage[colorlinks,citecolor=blue,pagebackref,hypertexnames=false]{hyperref}
\numberwithin{equation}{section}
\theoremstyle{plain}
\newtheorem{theorem}{Theorem}[section]

\newtheorem{lemma}[theorem]{Lemma}
\newtheorem{corollary}[theorem]{Corollary}
\newtheorem{proposition}[theorem]{Proposition}

\theoremstyle{definition}
\newtheorem{definition}[theorem]{Definition}

\newtheorem{example}[theorem]{Example}

\newtheorem{obs}{Observation}
\theoremstyle{remark}
\newtheorem{remark}[theorem]{Remark}

\newtheorem{case[theorem]}{Case}

\def \R{{\mathbb R}}

\def \F{{\mathcal F}}
\def \Z{{\mathbb Z}}

\def \T{{\mathbb T}}

\def\G{{\mathbb G}}
\def\H{{\mathbb H}}

\def\norm#1.#2.{\lVert#1\rVert_{#2}}

\def\R{\mathbb R}

\def \V{{\mathcal V}}

\def \W{{\mathcal W}}
\def \G{{\mathcal G}}
\def \H{{\mathcal H}}

\def \HS{{\mathcal{HS}}}

\title[Localization Operator and Weyl Transform on $\G$]
{Localization Operator and Weyl Transform on Reduced Heisenberg Group with Multi-dimensional Center
}

\author{Aparajita Dasgupta}\address{Aparajita Dasgupta  \endgraf Department of Mathematics	\endgraf Indian Institute of Technology Delhi	\endgraf New Delhi - 110 016, India.} \email{adasgupta@maths.iitd.ac.in}

\author{Santosh Kumar Nayak}\address{Santosh Kumar Nayak  \endgraf Department of Mathematics	\endgraf Indian Institute of Technology Delhi	\endgraf New Delhi - 110 016, India.} \email{nayaksantosh212@gmail.com, mathnayak@gmail.com}

\date{\today}
\begin{document}
	\thanks{The research of A. Dasgupta was supported by MATRICS Grant(RP03933G), Science and Engineering Research Board (SERB), DST, India and S. K. Nayak is supported by Institute fellowship.
	}
	
	\maketitle

	\allowdisplaybreaks
	
	\begin{abstract}
		In this article, we study two different types of operators, the localization operator and Weyl transform, on the reduced Heisenberg group with multidimensional center $\G$. The group $\G$ is a quotient group of non-isotropic Heisenberg group with multidimensional center $\H^m$ by its center subgroup.  Firstly, we define the localization operator using a wavelet transform on $\G$ and obtain the product formula for the localization operators. 
		%	We study the product of two localization operators and evaluate the product formula using a new convolution. We obtain some harmonic analysis results, like group Fourier transform, Plancherel formula, and inverse Fourier transform of the proposed group $G$.
		Next, we define the Weyl transform associated to the Wigner transform on $\G$ with the operator-valued symbol. Finally, we have shown that the Weyl transform is not only a bounded operator but also a compact operator when the operator-valued symbol is in $L^p,1\leq p\leq 2,$ and it is an unbounded operator when $p>2$.
	\end{abstract}

	%\tableofcontents 	
	
	\section{Introduction}
	In	\cite{Dau1}, I. Daubechies  introduced the localization operator using  the short-time Fourier transform.
	%when she studied time-frequency on $\R^n\times \widehat{\R^n}$.
	These operators are used to localize the signal on a time-frequency plane to get  information about the position and frequency in signal analysis, image processing, etc., (see \cite{Dau2,Dau1}). The localization operators are pseudo-differential operators associated with the so called anti-wick symbols and their applications can be found in the framework of PDEs and quantum mechanics, \cite{Mari,Kohn}.
	% Hence, their importance in various fields of mathematics and physics brought up the author's involvement. 
	In \cite{wong2,Liu} Wong et al. studied localization operators corresponding to wavelet transforms and proved that under suitable conditions on the symbols the operators are paracommutators, paraproducts and Fourier multipliers.
	
	%	 studied wavelet transform and polar wavelet transform on the affine group and polar affine group, respectively. They also defined the localization operator associated with wavelet transforms and polar wavelet transforms  with corresponding symbols.  

	Recall that the symbol of the composition of pseudo-differential operators on $\R^n$ can be expressed in terms of the asymptotic expansion. Similar kind of results for localization operators are obtained in \cite{Cor}, 
	%this operator on more general space like modulation space and obtained the boundedness and compactness of the operator. 
	%Also they found the formula of product of two localization operators in terms of symbols \cite{Du,Cor}. In \cite{Cor}, authors obtained the symbol of localization operator product in terms of asymptotic expansion and
	whereas in \cite{Du}, Wong et al., expressed the symbol of the product of localization operators in terms of some new convolution using the Weyl transform on $\R^n$. 
	
	Motivated by \cite{Du, Cor}, first we establish the localization operator with respect to the wavelet transform on the reduced Heisenberg group with multidimensional center, $\G$, and obtain the product formula of two localization operators. We expressed the product formula using a new convolution among the corresponding symbols.  To obtain this result the important tool is the $k$-Weyl transform which has been extensively studied in  \cite{shahla}.
	%	The following steps are followed:
	%	\begin{itemize}
		%		\item[(i)] Represent the two localization operators in terms of $k$-Weyl transforms;
		%		\item[(ii)] Then use the product formula for $k$-Weyl symbols  to compute the $k$-Weyl symbol of their product;
		%		\item[(iii)] Finally, express the obtained $k$-Weyl transform into a localization operator.
		%	\end{itemize}
	%	

	The classical Weyl transform was first introduced by H. Weyl in \cite{Weyl} to solve the quantization problem in quantum mechanics.  The Weyl transform  has been broadly used in the field of  PDEs, harmonic analysis, time-frequency analysis, etc., (see \cite{Simon,stien}). It is one of the powerful tool in solving problems like regularity, ellipticity, spectral asymptotic, etc., the reader can follow the book \cite{wong1} for their applications.
	
	Let $\sigma$ be a suitable function on $\R^{2n}$. Then the classical Weyl transform $\W_{\sigma}$ of $\sigma$ is an operator on the Schwartz space $S(\R^n)$  defined by $$(\W_{\sigma}f)(x)=(2\pi)^{-n}\int_{\R^{2n}}e^{i(x-y)\cdot \xi}\sigma(\frac{x+y}{2},\xi)f(y)dyd\xi,\quad f\in S(\R^n).$$
	Also we can represent the Weyl transform, $\W_{\sigma}$, in terms of Wigner transform, for any $f,g\in S(\R^n)$
	
	$$\langle\W_{\sigma}f,g\rangle_{L^2(\R^n)}=(2\pi)^{-n/2}\langle \W(f,g),\overline{\sigma}\rangle_{L^2(\R^{2n})},$$ where 
	$$\W(f,g)(x,\xi)=(2\pi)^{-n/2}\int_{\R^{n}}e^{-i\xi\cdot p}f(x+p/2)\overline{g(x-p/2)}dp,$$ for all $ x,y\in\R^n.$  Here we recall the following theorem, \cite{wong1}.
	\begin{theorem}
		Let $\sigma$ be in $L^2(\R^{2n})$. Then the Weyl transform $\W_{\sigma}:L^2(\R^n)\to L^2(\R^n)$ is a Hilbert-Schmidt operator.
	\end{theorem}
	Assume $\sigma\in L^p(\R^{2n}), 1\leq p\leq 2$, then the Weyl transform $\W_{\sigma}:L^2(\R^n)\to L^2(\R^n)$ is a compact operator. But for each $2< p< \infty$, there exists a symbol $\sigma\in L^p(\R^{2n})$ such that the Weyl transform $\W_{\sigma}:L^2(\R^n)\to L^2(\R^n)$ is not bounded operator. One can refer Wong's book \cite{wong1} for the proofs .  
	
	The above results for Weyl transform associated with the Wigner transform have been also studied in other settings like the Heisenberg group \cite{Peng}, upper half plane \cite{Peng1}, quaternion Heisenberg group \cite{Chen}. In \cite{Rachdi,Zhao}, authors defined the Weyl transform associated to Wigner transform using spherical mean operator and also studied their properties.  In Heisenberg group the group center, $Z(\H)$, coincide with the set
	$\{(0,0,t):t\in\R\}$. Here the periodicity in $t\in\R$ of the Schr\"{o}dinger representation, $\rho_{\lambda}$ can be inconvenient in the context of some applications. It is neither faithful nor square integrable. Therefore it is useful to occasionally use the so called reduced Heisenberg group,
	$$\H_{\text{red}}=\H/\{(0,0,t):t\in\R\}.$$ Here we consider the reduced Heisenberg group with multidimensional center,  denoted by $\G$. Our aim is to study the Weyl transform using the operator-valued symbol on the reduced Heisenberg group $\G$, and then investigate all aforementioned properties of Weyl transform for this group $\G$.

	Apart from the introduction the paper  is organized as follows. In Section \ref{s2}, we describe the square-integrable representation, group Fourier transform, Plancherel formula and inverse Fourier transform. In Section \ref{s3}, we define the localization operator and also find their connection with the $k$-Weyl transform. Then obtain the product formula of the localization operators. In the final section, we define the Weyl transform on the reduced Heisenberg group with multidimensional center associated to the Wigner transform and characterize  symbol spaces in terms of boundedness and unboundedness of the Weyl transform(operator).
	\section{Harmonic Analysis on Reduced Heisenberg group with multi-dimensional center}\label{s2}
	Let us fix $(c_1,c_2,\cdots,c_n)$ in $\R^n$. The nonisotropic Heisenberg group on $\R^n\times\R^n\times\R$ is defined by
	\begin{equation}\label{grouplaw}
		(x,y,t)\ast (x^{\prime},y^{\prime},t^{\prime})=\Big(x+x^{\prime},y+y^{\prime},t+t^{\prime}+\frac{1}{2}\sum_{j=1}^n c_j(x_jy_j^{\prime}-x_j^{\prime}y_j)\Big),
	\end{equation} for all $(x,y,t),(x^{\prime},y^{\prime},t^{\prime})$ in $\R^n\times \R^n\times \R$. Put $c_j=1$ for all $1\leq j\leq n$. Then we get the ordinary Heisenberg group (see \cite{thang}), with one dimensional center, 
	$Z=\{(0,0,t)\in \R^n\times\R^n\times\R:t\in\R\}$. One can observe that the group operation in \eqref{grouplaw} does not contain $x_ix_j^{\prime},i\neq j$. In \cite{shahla}, the author introduced the terms $x_ix_j^{\prime},i\neq j$, in \eqref{grouplaw} in a suitable manner by taking some orthogonal matrices in the group operation \eqref{grouplaw}. More precisely, the group is defined on $\R^n\times\R^n\times\R^m$
	with the operation 
	$$(x,y,t)\ast (x^{\prime},y^{\prime},t^{\prime})=\Big(x+x^{\prime},y+y^{\prime},t+t^{\prime}+\frac{1}{2}[z,z^{\prime}]\Big),$$ for all $(x,y,t),(x^{\prime},y^{\prime},t^{\prime})$ in $\R^n\times \R^n\times \R^m$, where $z=(x,y),z^{\prime}=(x^{\prime},y^{\prime})\in \R^{2n}$. Here $[z,z^{\prime}]$ in $\R^m$, is defined by 
	$$[z,z^{\prime}]_j=x^{\prime}\cdot B_j y-x\cdot B_j y^{\prime}, j=1,2,\cdots,m,$$ 
	where $B_1,B_2,\cdots,B_m$ are $n\times n$ orthogonal matrices such that 
	$$B_j^{-1}B_k=-B_k^{-1}B_j,j\neq k.$$
	We  denote the group by $\H^m$, and the dimension of the center $Z$, given by
	$$Z=\{(0,0,t)\in \R^n\times\R^n\times\R^m:t\in\R^m\},$$ is $m$. This group  $\H^m$ is known as nonisotropic Heisenberg group with multi-dimensional center. The relation between dimensions $n$ and $m$ is $m\leq n^2$, \cite[Proposition 2]{shahla}.
	The irreducible, unitary representations of $\H^m$, namely the Schr\"odinger representations on $L^2(\R^n)$, are given by 
	\begin{equation}
		\pi_{\lambda}(q,p,t)\varphi(x)=e^{i\lambda\cdot t}e^{iq\cdot B_{\lambda}(x+p/2)}\varphi(x+p), x\in\R^n,0\neq\lambda\in\R^{m\ast},
	\end{equation}
	for all $\varphi\in L^2(\R^n)$ and $(q,p,t)\in \R^n\times \R^n\times \R^m,$ where $B_{\lambda}=\lambda_1B_1+\cdots +\lambda_mB_m,$ and $\lambda=(\lambda_1,\lambda_2,\cdots,\lambda_m)\in\R^m$. It is easy to check that $$B_{\lambda}B_{\lambda}^{t}=|\lambda|^2I,$$ where $I$ is the identity $n\times n$ matrix, and $|\lambda|=\sqrt{\lambda_1^2+\cdots+\lambda_m^2}$. In particular $\det B_{\lambda}=|\lambda|^n$.
	
	%	Let $(x_0,y_0,t_0)\in Z(\H^m)$, center of $\H^m$. Then  
	%	$$(x_0,y_0,t_0)\ast (x,y,t)=(x,y,t)\ast (x_0,y_0,t_0),$$ for all $(x,y,t)\in \H^m$. 
	%Since $[z,z_0]=0$, then $x_0\cdot B_j y-x\cdot B_j y_0=0, 1\leq j\leq m.$ In particular for $x=x_0$ and for all $y\in\R^n$,$$\langle x_0,B_j(y-y_0)\rangle=0, 1\leq j\leq m.$$ So $\langle B_j^t x_0,y-y_0\rangle=0$, and hence $x_0=0$ and similarly by taking $y_0$ fixed in $\R^n$ and for all $x\in\R^n$, we get $y_0=0.$
	%	Hence we can obtain $\{(0,0,t_0): t_0\in\R^m\}$ is the centre of $\H^m$.
	Let $H=\{(0,0,2\pi t):t\in\Z^m\}$. It can be easily shown that, this is a normal subgroup of $\H^m$ under the operation $\ast$. Now define the quotient group $\G={\H^m}/{H}$ with the same operation $\ast$, $$(x,y,t)\ast (x^{\prime},y^{\prime},t^{\prime})=\Big(x+x^{\prime},y+y^{\prime},t+t^{\prime}+\frac{1}{2}[z,z^{\prime}]\Big),$$ where the coordinates of  $t+t^{\prime}+\frac{1}{2}[z,z^{\prime}]$ are modulo $2\pi$. We call the group, $\G=\R^n\times\R^n\times \T^m$, the reduced Heisenberg group with multidimensional center, where $\T^m$ is the $m$-dimensional torus. 
	
	The one-dimensional representations of the nonisotropic Heisenberg group $\H^m$ are still representations
	of $\G$ but the Schr\"odinger representations, $\pi_{\lambda}$ of $\H^m$ are  representations
	of $\G$ only when $0\neq\lambda\in\Z^m$.  For better understanding we refer \cite{thang}. These representations are relevant to the Plancherel theorem.

	\begin{definition}[Fourier transform]
		Let $f$ be in $L^1(\G)$	and $k\in \Z^{m\ast}$. Define the group Fourier transform $(\F f)(k)$ of $f$ at $k$	by
		$$(\F f)(k)\varphi=\int_{\G}f(q,p,t)\pi_k(q,p,t)\varphi d\mu(q,p,t),\quad \varphi\in L^2(\R^n),$$ where $d\mu(q,p,t)=\dfrac{dqdpdt}{(2\pi)^m}$.
	\end{definition}
	For $k\in \Z^{m}$, let us denote
	$$f^k(q,p):=\int_{\T^m}e^{ik\cdot t}f(q,p,t)\frac{dt}{(2\pi)^m},$$
	the inverse Fourier transform of $f$ in $t=(t_1,t_2,\cdots,t_m)$ on $\T^m$. So the group Fourier transform can be rewritten $$(\F f)(k)=\int_{\R^{2n}}f^k(q,p)\pi_{k}(q,p)dqdp,$$ where $\pi_k(q,p)\varphi(x)=e^{iq\cdot B_{k}(x+p/2)}\varphi(x+p), x\in\R^n.$
	Next we prove the Plancherel formula for $\G$.
	\begin{theorem}[Plancherel Theorem]\label{Plancherel}
		Let $f$ be in $L^2(\G)$ and $k\in\Z^{m\ast}$. Then $(\F f)(k):L^2(\R^n)\to L^2(\R^n)$ is a Hilbert-Schmidt operator and
		\begin{equation}
			\sum_{k\in\Z^{m\ast}}\|(\F f)(k)\|_{\HS}^2(2\pi)^{-n}\|k\|_2^n=\|f-f^{0}\|_{L^2(\G)}^2,
		\end{equation}
		where $\|k\|_2=\left(|k_1|^2+\cdots+|k_m|^2\right)^{1/2},$ and $$f^{0}(q,p,t)=\int\limits_{\T^m}f(q,p,t)\dfrac{dt}{(2\pi)^m}.$$ Furthermore, for $f,g\in L^2(\G)$, we have
		$$\sum_{k\in\Z^{m\ast}}\operatorname{tr}\left[(\F f)(k)(\F g)(k)^{\ast}\right](2\pi)^{-n}\|k\|_2^n=\langle f-f^{0},g-g^0\rangle_{L^2(\G)}.$$ 
	\end{theorem}
	\begin{proof}Let $\varphi$ be in $L^2(\R^n)$. Then by using the Fourier transform of $\R^n$, we get
		\begin{align*}
			(\F f)(k)\varphi(x)&=\int_{\R^{2n}}f^k(q,p)\pi_{k}(q,p)\varphi(x)dqdp\\
			&=\int_{\R^{2n}}f^k(q,p)e^{iq\cdot B_k(x+\frac{p}{2})}\varphi(x+p)dqdp\\
			&=\int_{\R^{n}}\left(\int_{\R^{n}}f^k(q,p-x)e^{iq\cdot B_k(\frac{x+p}{2})}dq\right)\varphi(p)dp\\
			&=\int_{\R^{n}}N_k(x,p)\varphi(p)dp,
		\end{align*}
		where $$N_k(x,p)=\int_{\R^{n}}f^k(q,p-x)e^{iq\cdot B_k(\frac{x+p}{2})}dq=(2\pi)^{n/2}\widehat{f^k}\left(-B_k\left(\frac{x+p}{2}\right),p-x\right).$$ By using the Plancherel theorem for the Fourier transform on $\R^{2n}$, we obtain
		\begin{align*}
			\|(\F f)(k)\|_{\HS}^2&=\|N_k(x,p)\|_{L^2(\R^{2n})}^2\\
			&=(2\pi)^n\int_{\R^{2n}}\left|\widehat{f^k}\left(-B_k\left(\frac{x+p}{2}\right),p-x\right)\right|^2dxdp\\
			&= (2\pi)^n\int_{\R^{2n}}\left|\widehat{f^k}\left(x,p\right)\right|^2 \|k\|_2^{-n}dxdp\\
			&=(2\pi)^n\|k\|_2^{-n}\|f^k\|_{L^2(\R^{2n})}^2<\infty.
		\end{align*}
		Using the Plancherel formula for $\T^m$, we get
		\begin{align}
			\sum_{k\in\Z^{m\ast}}\|(\F f)(k)\|_{\HS}^2(2\pi)^{-n}\|k\|_2^n&= \sum_{k\in\Z^{m\ast}}\|f^k\|^2_{L^2(\R^{2n})}\nonumber\\
			&=\int_{\R^{2n}}\sum_{k\in\Z^{m\ast}}|f^k(q,p)|^2dqdp\label{p}\\
			&=\int_{\R^{2n}}\sum_{k\in\Z^{m}}|h^k(q,p)|^2dqdp\nonumber \\
			&=\int_{\R^{2n}}\int_{\T^m}|h(q,p,t)|^2\dfrac{dqdpdt}{(2\pi)^m}=\|f-f^0\|_{L^2(\G)}^2,
		\end{align}
		where $h(q,p,t)=f(q,p,t)-f^0(q,p,t)$, and then $h\in L^2(\G)$ and 
		$$h^k(q,p)=
		\begin{cases}
			f^k(q,p), &  \text{for}\quad k\neq 0\\
			0, & \text{for} \quad k=0.
		\end{cases}$$
	\end{proof}
	\begin{remark}
		If we add $\int\limits_{\R^{2n}}|f^0(q,p)|^2dqdp$ in \eqref{p}, we obtain the exact analogue of Plancherel formula
		\begin{equation}
			\sum_{k\in\Z^{m\ast}}\|(\F f)(k)\|_{\HS}^2(2\pi)^{-n}\|k\|_2^n+\int_{\R^{2n}}|\widehat{(f^0)}(x,y)|^2dxdy=\int_{\G}|f(q,p,t)|^2\frac{dqdpdt}{(2\pi)^m}.
		\end{equation}
	\end{remark}
	In the next theorem we derive the inverse Fourier transform.
	\begin{theorem}[Inverse Fourier transform]
		Let $f$ be in the Schwartz space $S(\G)$. Then the inversion formula is given by
		$$f(q,p,t)=\sum_{k\in\Z^{m\ast}}\operatorname{tr}\left(\pi_{k}(q,p,t)^{\ast}(\F f)(k)\right)(2\pi)^{-n}\|k\|_2^n+f^0(q,p,t).$$
		Moreover,
		\begin{align*}
			f(q,p,t)=\sum_{k\in\Z^{m\ast}}\operatorname{tr}\left(\pi_{k}(q,p,t)^{\ast}(\F f)(k)\right)(2\pi)^{-n}\|k\|_2^n+(2\pi)^{-n}\int_{\R^{2n}}e^{i(q\cdot x+q\cdot y)}\widehat{(f^0)}(x,y)dxdy.
		\end{align*}
	\end{theorem}
	\begin{proof}
		For all $(q,p,t)$ in $\G$ and $k\in\Z^{m\ast}$,
		\begin{equation}\label{g}
			\pi_k(q,p,t)^{\ast}(\F f)(k)=\int_{\T^m}\int_{\R^{2n}}g(q^{\prime},p^{\prime},t^{\prime})\pi_k(q^{\prime},p^{\prime},t^{\prime})d\mu(q^{\prime},p^{\prime},t^{\prime}),
		\end{equation} where $g_{(q,p,t)}(q^{\prime},p^{\prime},t^{\prime})=e^{-i\frac{k}{2}\cdot[z,z^{\prime}]}f(q^{\prime}+q,p^{\prime}+p,t^{\prime}+t)$ and $z=(q,p),z^{\prime}=(q^{\prime},p^{\prime}).$ Then \eqref{g} can rewrite as
		$$\pi_k(q,p,t)^{\ast}(\F f)(k)=(\F g_{(q,p,t)})(k).$$ The kernel of $(\F g_{(q,p,t)})(k)$ is given by
		$$N_k^{(q,p,t)}(x,y)=(2\pi)^{n/2}\widehat{g_{(q,p,t)}^k}\left(-B_k\left(\frac{x+y}{2}\right),y-x\right).$$ Also
		\begin{align*}
			\text{tr}\left(\pi_k(q,p,t)^{\ast}(\F f)(k)\right)=\int_{\R^{n}}N_k^{(q,p,t)}(x,x)dx&=(2\pi)^{n/2}\int_{\R^{n}}\widehat{g^k_{(q,p,t)}}\left(-B_k x,0\right)dx\\
			&=\int_{\R^{2n}}e^{iy\cdot B_kx}g^k_{(q,p,t)}(y,0)dy.
		\end{align*}
		On the other hand, we can find
		$$g_{(q,p,t)}^k(q^{\prime},p^{\prime})=e^{-i\frac{k}{2}\cdot[z,z^{\prime}]}e^{-ik\cdot t}f^{k}(q^{\prime}+q,p^{\prime}+p).$$ Then $$g_{(q,p,t)}^k(y,0)=e^{-i\frac{y}{2}\cdot B_kp}e^{-ik.t}f^k(y+q,p).$$
		Thus
		\begin{align*}
			\text{tr}\left(\pi_k(q,p,t)^{\ast}(\F f)(k)\right)&=\int_{\R^{2n}}e^{iy\cdot B_kx}e^{-i\frac{y}{2}\cdot B_kp}e^{-ik.t}f^k(y+q,p)dxdy\\
			&=e^{-ik.t}\int_{\R^{2n}}e^{i(y-q)\cdot B_kx}e^{-i\frac{y-q}{2}\cdot B_kp}f^k(y,p)dxdy\\
			&=e^{-ik.t} e^{iq\cdot B_k\frac{p}{2} }\int_{\R^{2n}}e^{-iq\cdot B_kx}e^{iy\cdot \left(B_kx-B_k\frac{p}{2}\right)}f^k(y,p)dxdy\\
			&=(2\pi)^{n/2}e^{-ik\cdot t} e^{iq\cdot B_k\frac{p}{2} }\int_{\R^{n}}e^{-iq\cdot B_kx}\left(\F_1^{-1}f^k\right)\left(B_kx-B_k\frac{p}{2},p\right)dx\\
			&=(2\pi)^{n/2}e^{-ik.\cdot t}\int_{\R^{n}}e^{-iq\cdot x}\left(\F_1^{-1}f^k\right)(x,p)\frac{dx}{\|k\|_2^n}\\
			&=(2\pi)^n\|k\|_2^{-n}e^{-ik\cdot t}f^k(q,p).
		\end{align*}
		By using the inverse Fourier transform on $\T^m$, we get
		\begin{align*}
			\sum_{k\in\Z^{m\ast}}\text{tr}\left[\pi_k^{\ast}(q,p,t)(\F f)(k)\right](2\pi)^{-n}\|k\|_2^n&= \sum_{k\in\Z^{m\ast}}e^{-ik\cdot t}f^k(q,p)\\
			&= \sum_{k\in\Z^m}e^{-ik\cdot t}h^k(q,p)\\
			&= h(q,p,t)=(f-f^0)(q,p,t).
		\end{align*}
	\end{proof}

	The next theorem shows that all the Schrödinger representations $\pi_k$, $k\in\Z^{m\ast}$, are square integrable representations.
	\begin{theorem}
		For $k=(k_1,k_2,\cdots,k_m)\in\Z^{m\ast}$,
		$$\pi_{k}(q,p,t)\varphi(x)=e^{ik\cdot t}e^{iq\cdot B_{k}(x+p/2)}\varphi(x+p), x\in\R^n$$
		for all $\varphi\in L^2(\R^n)$ and $(q,p,t)\in \G,$ where $B_{k}=k_1B_1+\cdots +k_mB_m,$ are square integrable representations.
	\end{theorem} 
	\begin{proof} Using the Plancherel theorem for $\R^n$, we get
		\begin{align*}
			&\int_{\G}|\langle\varphi,\pi_{k}(q,p,t)\rangle_{L^2(\R^n)}|^2 dpdqdt\\
			&=\int_{[0,2\pi]^m}\int_{\R^n}\int_{\R^n}\Big\lvert\int_{\R^n}\varphi(x)e^{-ik\cdot t}e^{-iq\cdot B_{k}(x+p/2)}\overline{\phi(x+p)}dx\Big\rvert^2dpdqdt\\
			& =(2\pi)^m\int_{\R^n}\int_{\R^n}\Big\lvert\int_{\R^n}\varphi(x)e^{-iq\cdot B_{k}(x+p/2)}\overline{\phi(x+p)}dx\Big\rvert^2dpdq\\
			& = (2\pi)^m\int_{\R^n}\int_{\R^n}\Big\lvert\int_{\R^n}(\varphi T_{-p}\overline{\varphi})(x)e^{-i B_{k}^{t}(q)\cdot x}e^{-iB_{k}(p/2)}dx\Big\rvert^2dpdq\\
			&=  (2\pi)^{m+n}\int_{\R^n}\int_{\R^n}\Big\lvert\widehat{(\varphi T_{-p}\overline{\varphi})}( B_{k}^{t}q)\Big\rvert^2dpdq\\
			& =|\det(B_{k})|^{-1}(2\pi)^{m+n}\|\varphi\|_{L^2(\R^n)}^4<\infty,
		\end{align*}
		where $(\varphi T_{-p}{\varphi})(x)=\varphi(x)\varphi(x+p).$
	\end{proof}
	
	%\textcolor{red}{Write here about $k$-Weyl transform}
	Let $f,g$ be in $L^2(\R^n)$. Recall the $k$-Fourier Wigner transform $\V_k(f,g)$ of $f$ and $g$ on $\R^n\times\R^n$, are defined by
	$$\V_k(f,g)(q,p)=(2\pi)^{-n/2}\langle\pi_k(q,p)f,g\rangle_{L^2(\R^n)}.$$ We can relate the ordinary Fourier Wigner transform $\V(f,g)$ with the $k$-Fourier Wigner transform $\V_k(f,g)$ by
	$$\V_k(f,g)(q,p)=\V(f,g)(B_k^{t}q,p).$$
	Then the $k$-Wigner transform $\W^k(f,g)$ of $f$ and $g$ is defined as the Fourier transform of $k$-Fourier Wigner transform, i.e.,
	$$\W^k(f,g)=\widehat{\V_k(f,g)},$$ for all $f,g\in L^2(\R^n)$. 
	Again one can relate the $k$-Wigner transform $\W^k(f,g)$ with ordinary Wigner transform $\W(f,g)$ by
	$$\W^k(f,g)(x,\xi)=\|k\|_2^n\W(f,g)\left(\frac{1}{|k|^2}B_k^t x,\xi\right),$$ for all $x,\xi\in \R^n$ and $f,g\in L^2(\R^n)$.
	
	Let $\sigma$ be in the Schwartz space, $S(\R^n\times\R^n)$, of $\R^n\times\R^n$. Then the $k$-Weyl transform $\W^k_{\sigma}$, of $\sigma$ associated to the $k$-Wigner transform $\W^k$, is defined by
	$$\langle \W_{\sigma}^{k}f,g\rangle_{L^2(\R^n)}=(2\pi)^{-n/2}\int_{\R^{2n}}\sigma(x,\xi)\W^k(f,g)(x,\xi)dxd\xi,$$ for all $f,g\in S(\R^n).$ The $k$-Weyl transform $\W^k_{\sigma}$ can also be written as
	$$\langle \W_{\sigma}^{k}f,g\rangle_{L^2(\R^n)}=(2\pi)^{-n/2}\int_{\R^{2n}}\widehat{\sigma}(q,p)\V_k(f,g)(q,p)dqdp,$$ for all $f,g\in S(\R^n),$ using the Parseval identity for the Euclidean Fourier transform.
	Furthermore, the $k$-Weyl transform $\W_{\sigma}^k$ can also be expressed as
	$$(\W_{\sigma}^kf)(x)=(2\pi)^{-n}\int_{\R^{2n}}\widehat{\sigma}(q,p)\pi_k(q,p)f(x)dqdp.$$
	
	For detailed study of the $k$-Fourier Wigner transform, the $k$-Wigner transform and the $k$-Weyl transform, we refer to \cite{shahla}.
	
	Next we express the group Fourier transform $\F f$ in terms of $k$-Weyl transform. 
	
	\begin{proposition}
		Let $f\in L^1(\G)$. Then for all $k\in\Z^{m\ast}$
		$$(\F f)(k)=(2\pi)^n \W^k_{{\widecheck{(f^k)}}},$$ where $\widecheck{(f^k)}$ is the inverse Fourier transform of $f^k$ on $\R^{2n}$.
	\end{proposition}
	\begin{proof}
		For proof, we refer to \cite{shahla}.
	\end{proof}
	\section{Localization Operator on Reduced Heisenberg group with Multi-dimensional Center}\label{s3}
	The resolution of identity is given by
	$$\langle f,g\rangle_{L^2(\R^n)}=\frac{1}{c_{\varphi}}\int_{\G}\langle f,\pi_{k}(q,p,t)\varphi\rangle_{L^2(\R^n)}\langle\pi_{k}(q,p,t)\varphi,g\rangle_{L^2(\R^n)}dqdpdt,$$ for all $f,g$ in $L^2(\R^n)$ and an admissible wavelet $\varphi\in L^2(\R^n)$. For any admissible wavelet $\varphi$, it is easy to see that $$c_{\varphi}=(2\pi)^{m+1}\|k\|_2^{-n}.$$ The localization operator is given by
	$$\langle L_{\sigma,\varphi}^{k}f,g\rangle_{L^2(\R^n)}=\frac{1}{c_{\varphi}}\int_{\G}\sigma(q,p,t)\langle f,\pi_{k}(q,p,t)\varphi\rangle_{L^2(\R^n)}\langle\pi_{k}(q,p,t)\varphi,g\rangle_{L^2(\R^n)}dqdpdt,$$ for all $f,g$ in $L^2(\R^n)$ and an admissible wavelet $\varphi\in L^2(\R^n)$. Now put $\sigma(q,p,t)=F(q,p)$ and $\pi_{k}(q,p,t)=e^{ik\cdot t}\pi_{k}(q,p)$, where $\pi_{k}(q,p)\varphi(x)=e^{iq\cdot B_{k}(x+p/2)}\varphi(x+p).$ Then the localization operator with the symbol $F$ becomes
	\begin{align*}
		\langle L_{\sigma,\varphi}^{k}f,g\rangle_{L^2(\R^n)}
		& =\frac{1}{c_{\varphi}}\int_{\G}F(q,p)\langle f,\pi_{k}(q,p,t)\varphi\rangle_{L^2(\R^n)}\langle\pi_{k}(q,p,t)\varphi,g\rangle_{L^2(\R^n)}dqdpdt\\
		& =(2\pi)^{-m-n}\|k\|_{2}^n (2\pi)^m \int_{\R^{2n}}F(q,p)\langle f,\pi_{k}(q,p)\varphi\rangle_{L^2(\R^n)}\langle\pi_{k}(q,p)\varphi,g\rangle_{L^2(\R^n)}dqdp\\
		& = (2\pi)^{-n}\|k\|_{2}^n \int_{\R^{2n}}F(q,p)\langle f,\pi_{k}(q,p)\varphi\rangle_{L^2(\R^n)}\langle\pi_{k}(q,p)\varphi,g\rangle_{L^2(\R^n)}dqdp
	\end{align*} 
	Hence we have the localization operator as, 
	\begin{equation}\label{Localization}
		(L^{k}_{F,\varphi}f)(x)=(2\pi)^{-n}\|k\|_{2}^n \int_{\R^{2n}}F(q,p)\langle f,\pi_{k}(q,p)\varphi\rangle_{L^2(\R^n)}\pi_{k}(q,p)\varphi(x)dqdp, x\in\R^n.
	\end{equation} Let $f_{k,x}(q,p)=\langle f,\pi_{k}(q,p)\varphi\rangle_{L^2(\R^n)}\pi_{k}(q,p)\varphi(x)$, with $ \varphi(x)=2^{\frac{n}{4}}e^{-\frac{|x|^2}{2}}$ and $k \in\Z^{m\ast}$.
	\begin{lemma}
		The Fourier transform of $f_{k,x}(q,p)$ is
		$$\widehat{f_{k,x}}(\xi,\eta)=\dfrac{(2\pi)^{n/2}}{\det B_{k}}\times \pi_{k}\Big(B_{k}(\frac{\eta}{\|k\|_2^2}),-B_{k}^{-1}\xi\Big)f(x)\V^{k}(\varphi,\varphi)\Big(-B_{k}(\frac{\eta}{\|k\|_2^2}),B_{k}^{-1}\xi\Big).$$
	\end{lemma} 
	\begin{proof}
		We will show  $\lim_{\epsilon\to 0}I_{\epsilon}(\xi,\eta)=\widehat{f_x(\xi,\eta)}$, where 	$$I_{\epsilon}(\xi,\eta) =(2\pi)^{-n}\int_{\R^{2n}}e^{-i(q\cdot\xi+p\cdot\eta)}e^{-\epsilon^2\frac{|q|^2}{2}}\langle f,\pi_{k}(q,p)\varphi\rangle_{L^2(\R^n)}\pi_{k}(q,p)\varphi(x)dqdp.$$
		Now 
		\begin{align*}
			I_{\epsilon}(\xi,\eta) &=(2\pi)^{-n}\int_{\R^{2n}}e^{-i(q\cdot\xi+p\cdot\eta)}e^{-\epsilon^2\frac{|q|^2}{2}}\langle f,\pi_{k}(q,p)\varphi\rangle_{L^2(\R^n)}\pi_{k}(q,p)\varphi(x)dqdp\\
			&= (2\pi)^{-n}\int_{\R^{2n}}\Big(\int_{\R^n}e^{-iq\cdot(\xi+B_{k}(y+p/2)-B_{k}(x+p/2))}e^{-\epsilon^2\frac{|q|^2}{2}} dq\Big)e^{-ip\cdot\eta}f(y)\overline{\varphi(y+p)}\varphi(x+p)dydp\\
			&=(2\pi)^{-n/2}\int_{\R^{2n}}\epsilon^{-n}e^{\frac{|\xi+B_{k}(y-x)|^2}{2\epsilon^2}}e^{-ip\cdot\eta}f(y)\overline{\varphi(y+p)}\varphi(x+p)dydp\\
			&= (2\pi)^{-n/2}\int_{\R^{n}}e^{-ip\cdot\eta}\varphi(x+p)\Big(\int_{\R^{n}}\epsilon^{-n}e^{\frac{|\xi+B_{k}(y-x)|^2}{2\epsilon^2}}f(y)\overline{\varphi(y+p)}dy\Big)dp\\
			&= (2\pi)^{-n/2}\int_{\R^{n}}e^{-ip\cdot\eta}\varphi(x+p)(F_p\ast \psi_{\epsilon})(x-B_{k}^{-1}\xi)dp,
		\end{align*}
		where $\psi(x)=e^{\frac{-|B_{k}x|^2}{2}}, \psi_{\epsilon}(x)=\epsilon^{-n}\psi(\frac{x}{\epsilon}),F_p(y)=f(y)\overline{\varphi(y+p)}$. Now  $$F_p\ast\psi_{\epsilon}\rightarrow F_p\int_{\R^{n}}\psi(x)dx=F_p\int_{\R^{n}}e^{\frac{-|B_{k}x|^2}{2}}dx=\dfrac{(2\pi)^{n/2}}{\det(B_{k})}F_p,$$ as $\epsilon\rightarrow 0$. By Lebesgue dominated convergence theorem, 
		$$\lim_{\epsilon\to 0}I_{\epsilon}(\xi,\eta)=\dfrac{1}{\det(B_{k})}\int_{\R^{n}}e^{-ip\cdot\eta}\varphi(x+p)F_p(x-B_{k}^{-1}\xi)dp.$$
		Now
		\begin{align*}
			\widehat{f_x(\xi,\eta)} &=\dfrac{f(x-B_{k}^{-1}\xi)}{\det(B_{k})}\int_{\R^n}e^{-ip\cdot\eta}\varphi(x+p)\overline{\varphi(x-B_{k}^{-1}\xi+p)}dp\\
			&=\dfrac{f(x-B_{k}^{-1}\xi)}{\det(B_{k})}\int_{\R^n}e^{-i(p-x)\cdot\eta}\varphi(p)\overline{\varphi(p-B_{k}^{-1}\xi)}dp\\
			&=\dfrac{e^{ix\cdot\eta}f(x-B_{k}^{-1}\xi)}{\det(B_{k})}\int_{\R^n}e^{-ip\cdot\eta}\varphi(p)\overline{\varphi(p-B_{k}^{-1}\xi)}dp\\
			&= \dfrac{e^{ix\cdot\eta}f(x-B_{k}^{-1}\xi)}{\det(B_{k})}\int_{\R^n}\exp\Big({-i\big(p+B_{k}^{-1}\dfrac{\xi}{2}\big)\cdot\eta}\Big)\varphi\Big(p+\dfrac{B_{k}^{-1}\xi}{2}\Big)\overline{\varphi\Big(p-B_{k}^{-1}\frac{\xi}{2}\Big)}dp\\
			&= (2\pi)^{n/2}\dfrac{f(x-B_{k}^{-1}\xi)}{\det(B_{k})}e^{i \eta\cdot (x-B_{k}^{-1}\frac{\xi}{2})}\V_{k}(\varphi,\varphi)\Big(-B_{k}(\frac{\eta}{\|k\|_2^2}),B_{k}^{-1}\xi\Big)\\
			&= \dfrac{(2\pi)^{n/2}}{\det B_{k}}\times \pi_{k}\Big(B_{k}(\frac{\eta}{\|k\|_2^2}),-B_{k}^{-1}\xi\Big)f(x)\V_{k}(\varphi,\varphi)\Big(-B_{k}(\frac{\eta}{\|k\|_2^2}),B_{k}^{-1}\xi\Big).
		\end{align*}
	\end{proof}
	Let $F$ and $G$ be functions in $L^2(\R^{2n})$. Then the $k$-twisted convolution of $F$ and $G$ denoted by $F\ast_{k}G$ on $\R^{2n}$, is defined by $$F\ast_{k}G(\xi,\eta)=\int_{\R^{2n}}F(\xi-q,\eta-q)G(q,p)e^{\frac{i}{2}k\cdot\left[(\xi,\eta),(q,p)\right]}dqdp,$$ 
	where $\left[(\xi,\eta),(q,p)\right]_j=q\cdot B_j \eta-\xi\cdot B_j p, \quad j=1,2,\cdots,m.$
	\begin{theorem}(see \cite{shahla})\label{productweyl}
		Let $\sigma$ and $\tau$ be in $L^2(\R^{2n})$. Then $$\W_{\sigma}^{k}\W_{\tau}^{k}=\W_{\gamma}^{k},$$ where $\widehat{\gamma}=(2\pi)^{-n}(\widehat{\sigma}\ast_{k}\widehat{\tau}).$
	\end{theorem}
	
	In  the next theorem, we identify the localization operator with the $k$-Weyl transform.
	\begin{theorem}\label{LocalizationWeyl}
		Let  $\varphi(x)=2^{n/4}e^{-\frac{|x|^2}{2}}$ be the admissible wavelet. Then for all $k\in\Z^{m\ast}$ and $F\in L^2(\R^{2n})$, the localization operator $L_{F,\varphi}^{k}$ can be represented as	$$L_{F,\varphi}^{k}=\W_{\sigma}^{k},$$ where $\sigma=F^{k}\ast \Lambda^{k}$, $\Lambda^{k}(q,p)=\dfrac{2^n}{\det(B_{k})}e^{-\left|\big((B_{k}^{t})^{-1}q,p\big)\right|^2},$ and  $F^{k}(q,p)=F\Big(B_{k}\dfrac{p}{\|k\|_2^2},-B_{k}^{t}\dfrac{q}{\|k\|_2^2}\Big)$.
	\end{theorem}
	\begin{proof}
		Using \eqref{Localization} and Parseval identity for Euclidean Fourier transform, we get	
		\begin{align*}
			&(L_{F,\varphi}^{k}f)(x)\\
			& =(2\pi)^{-n}\|k\|_{2}^n \int_{\R^{2n}}F(q,p)\langle f,\pi_{k}(q,p)\varphi\rangle_{L^2(\R^n)}\pi_{k}(q,p)\varphi(x)dqdp\\
			&= (2\pi)^{-n}\|k\|_{2}^n \int_{\R^{2n}}\widecheck{F}(\xi,\eta)\widehat{f_{k,x}}(\xi,\eta)d\xi d\eta\\
			& =\dfrac{(2\pi)^{-n/2}\|k\|_{2}^n}{\det(B_{k})}\int_{\R^{2n}}\widecheck{F}(\xi,\eta)\pi_{k}\Big(B_{k}(\frac{\eta}{\|k\|_2^2}),-B_{k}^{-1}\xi\Big)f(x)\V^{k}(\varphi,\varphi)\Big(-B_{k}(\frac{\eta}{\|k\|_2^2}),B_{k}^{-1}\xi\Big)d\xi d\eta\\
			& = {(2\pi)^{-n/2}}\int_{\R^{2n}}\widecheck{F}\Big(-B_{k}p,B_{k}^{t}q\Big)\pi_{k}(q,p)f(x)\V^{k}(\varphi,\varphi)(-q,-p)\det(B_{k})^2dp dq\\
			&= {(2\pi)^{-n/2}\det(B_{k})^2}\int_{\R^{2n}}\widehat{F}(B_{k}p,-B_{k}^{t}q)\pi_{k}(q,p)f(x)\V^{k}(\varphi,\varphi)(-q,-p)dp dq.
		\end{align*}
		Let us choose
		$\varphi(x)=2^{n/4}e^{-\frac{|x|^2}{2}}$, then
		\begin{align*}
			\V^{k}(\varphi,\varphi)(-q,-p)&=(2\pi)^{-n/2}\int_{\R^{n}}e^{-iB_{k}^{t}q\cdot x}\varphi(x-p/2)\overline{\varphi(x+p/2)}dx\\
			&= (2\pi)^{-n/2}2^{n/2}\int_{\R^{n}}e^{-iB_{k}^{t}q\cdot x}e^{-\frac{1}{2}(|x-p/2|^2+|x+p/2|^2)}dx\\
			&= (2\pi)^{-n/2}2^{n/2}\int_{\R^{n}}e^{-iB_{k}^{t}q\cdot x}e^{-(|x|^2+|\frac{p}{2}|^2)}dx\\
			& =(2\pi)^{-n/2}2^{n/2}e^{-\frac{|p|^2}{4}}\int_{\R^{n}}e^{-iB_{k}^{t}q\cdot x}e^{-|x|^2}dx\\
			&=e^{-\frac{|p|^2}{4}}e^{\frac{-|B_{k}^{t}q|^2}{4}}\\
			&=e^{-\frac{|p|^2}{4}}e^{-\frac{\|k\|_2^2|q|^2}{4}}.
		\end{align*}
		Let us denote $F^{k}(q,p)=F\Big(B_{k}\dfrac{p}{\|k\|_2^2},-B_{k}^{t}\dfrac{q}{\|k\|_2^2}\Big)$. Then
		\begin{align*}
			\widehat{F^{k}}(x,y)&=(2\pi)^{-n}\int_{\R^{2n}}e^{-i(x\cdot q+y\cdot p)}F\Big(B_{k}\dfrac{p}{\|k\|_2^2},-B_{k}^{t}\dfrac{q}{\|k\|_2^2}\Big)dqdp\\
			&= (2\pi)^{-n}\int_{\R^{2n}}e^{-i\left(-x\cdot B_{k}\eta+y\cdot B_{k}^{t}\xi\right)}F(\xi,\eta)\|k\|_2^{2n}d\xi d\eta \\
			&= \|k\|_2^{2n}\widehat{F}(B_{k}y,-B_{k}^{t}x).
		\end{align*}
		Hence $\widehat{F^{k}}(q,p)=\|k\|_2^{2n}\widehat{F}(B_{k}p,-B_{k}^{t}q)$.
		
		Assume $\Lambda^{k}(q,p)=\dfrac{2^n}{\det(B_{k})}e^{-\left|\big((B_{k}^{t})^{-1}q,p\big)\right|^2}$, then we get
		$$\widehat{\Lambda^{k}}(q,p)=e^{-\frac{|p|^2}{4}-\frac{\|k\|_2^2|q|^2}{4}}.$$
		Thus the localization operator can be identified by a $k$-Weyl transform, given by 
		\begin{align*}
			(L_{F,\varphi}^{k}f)(x)&=(2\pi)^{-n/2}\int_{\R^{2n}}\widehat{F^{k}}(q,p)\pi_{k}(q,p)f(x)\widehat{\Lambda^{k}}(q,p)dqdp\\
			&= (\W_{F^{k}\ast \Lambda^{k}}^{k}f)(x).
		\end{align*}
	\end{proof}
	Define a new convolution of $F$ and $G$ in $L^2(\R^{2n})$ by
	\begin{equation}\label{newconvolution}
		(F\circledast G)(\xi,\eta)=\int_{\R^{2n}}F(\xi-q,\eta-p)G(q,p)e^{-\frac{|p|^2}{2}-\frac{\|k\|_2^2 |q|^2}{2}}\times e^{\frac{\langle \eta, p\rangle}{2}+\frac{\|k\|_2^2\langle \xi,q\rangle}{2}}\times e^{\frac{i}{2}k\cdot\left[(\xi,\eta),(q,p)\right]}dqdp,
	\end{equation}
	whenever the integral exists, where $\langle\cdot ,\cdot\rangle$ and $|\cdot|$ denote the usual inner product and norm on $\R^n$, respectively, and $\|k\|_2=\sqrt{k_1^2+\cdots+k_m^2}$ for $k_i\in\Z$.  \\
	
	In the next theorem, we express the product formula of two localization operators using the new convolution defined in \eqref{newconvolution}, and show that the product is again a localization operator.
	\begin{theorem}
		Let $F$ and $G$ be in $L^2(\R^{2n})$, and $\varphi(x)=2^{n/4}e^{-\frac{|x|^2}{2}}$. If there exists a function $H$ in $L^2(\R^{2n})$ such that the localization operator $L^k_{H,\varphi}:L^2(\R^n)\to L^2(\R^n)$ is same as the product of localization operators $L^k_{F,\varphi}:L^2(\R^n)\to L^2(\R^n)$ and $L^{k}_{G,\varphi}:L^2(\R^n)\to L^2(\R^n)$, then 
		$$\widehat{H^{k}}(\xi,\eta)=(2\pi)^{-n}\Big(\widehat{F^{k}}\circledast\widehat{G^{k}}\Big)(\xi,\eta)$$
	\end{theorem}
	\begin{proof}
		Here $L_{F,\varphi}^{k}L_{G,\varphi}^{k}=L_{H,\varphi}^{k}$, then by the Theorem \ref{LocalizationWeyl}, we get
		$$\W_{F^{k}\ast \Lambda^{k}}^{k}\W_{G^{k}\ast \Lambda^{k}}^{k}=\W_{H^{k}\ast \Lambda^{k}}^{k}.$$ 
		Using Theorem \ref{productweyl} and \eqref{newconvolution}, we have 
		\begin{align*}
			\widehat{H^{k}}(\xi,\eta)e^{-\frac{1}{4}(|\eta|^2+\|k\|_2^2|\xi|^2)}&=(2\pi)^{-n}\int_{\R^{2n}}\widehat{F^{k}}(\xi-q,\eta-p)e^{-\frac{1}{4}\left(|\eta-p|^2+\|k\|_2^2|\xi-q|^2\right)}\widehat{G^{k}}(q,p)\\
			&\times e^{-\frac{1}{4}\left(|p|^2+\|k\|_2^2|q|^2\right)}e^{\frac{i}{2}k\cdot\left[(\xi,\eta),(q,p)\right]}dqdp.
		\end{align*}
		Then 
		\begin{align*}
			\widehat{H^{k}}(\xi,\eta)&=(2\pi)^{-n}\int_{\R^{2n}}\widehat{F^{k}}(\xi-q,\eta-p)\widehat{G^{k}}(q,p)e^{-\frac{1}{2}\left(|p|^2+\|k\|_2^2|q|^2\right)}\\
			& \times e^{\frac{1}{2}\left(\langle \eta,p\rangle+\|k\|_2^2\langle\xi,q\rangle\right)}\times e^{\frac{i}{2}k\cdot\left[(\xi,\eta),(q,p)\right]}dqdp\\
			& =(2\pi)^{-n}\Big(\widehat{F^{k}}\circledast\widehat{G^{k}}\Big)(\xi,\eta).
		\end{align*}
	\end{proof}
	%\begin{corollary}
	%	Let $F$ and $G$ be the function in $L^2(\R^{2n})$ such that 
	%\end{corollary}
	The following example shows that the new convolution $\circledast$ of two $L^2(\R^n)$ functions is not in $L^2(\R^n)$.
	\begin{example}
		For each $k\in\Z^{m\ast}$, there are two functions $F^k$ and $G^k$ in $L^2(\R^{2n})$ such that $\widehat{F^k}\circledast \widehat{G^k} $ is not in $L^2(\R^{2n})$. Let $F^k$ and $G^k$ be in $L^2(\R^{2n})$ such that 
		$$\widehat{F^k}(q,p)=e^{\frac{\|k\|_2^2|q|+|p|^2}{4}}\chi_{A}(q,p), \widehat{G^k}(q,p)=e^{\frac{-\|k\|_2^2}{2}|q|^2}\chi_B(p),$$ where $A$ and $B$ are compact subsets of $\R^{2n}$ and $\R^n$ respectively.
		Then 
		\begin{align*}
			&\widehat{F^k}\circledast\widehat{G^k}(\xi,\eta)\\
			&= \int_{\R^{2n}}\widehat{F^k}(\xi-q,\eta-p)\widehat{G^k}(q,p)e^{-\frac{|p|^2}{2}-\frac{\|k\|_2^2 |q|^2}{2}}\times e^{\frac{\langle \eta, p\rangle}{2}+\frac{\|k\|_2^2\langle \xi,q\rangle}{2}}\times e^{\frac{i}{2}k\cdot\left[(\xi,\eta),(q,p)\right]}dqdp\\
			&=\int_{\R^{2n}}e^{\frac{\|k\|_2^2|\xi-q|^2+|\eta-p|^2}{4}}e^{\frac{-\|k\|_2^2}{2}|q|^2}e^{-\frac{|p|^2}{2}-\frac{\|k\|_2^2 |q|^2}{2}}\times e^{\frac{\langle \eta, p\rangle}{2}+\frac{\|k\|_2^2\langle \xi,q\rangle}{2}}\times e^{\frac{i}{2}k\cdot\left[(\xi,\eta),(q,p)\right]}dqdp\\
			&=e^{\frac{\|k\|_2^2|\xi|^2+|\eta|^2}{4}}\int_{\R^{2n}}e^{-\frac{3\|k\|_2^2}{4}|q|^2}e^{-\frac{1}{4}|p|^2}e^{\frac{i}{2}(q\cdot B_k\eta-\xi\cdot B_kp)}dqdp\\
			&=c e^{\frac{\|k\|_2^2|\xi|^2+|\eta|^2}{4}}e^{-\frac{|B_k^t \xi|^2}{4}} e^{-\frac{|B_k\eta|^2}{12\|k\|^2}}=ce^{\frac{|\eta|^2}{6}},
		\end{align*}  
		where $c$ is a constant depending on $k$ and $n$. It is easy to check that $\widehat{F^k}\circledast \widehat{G^k}$ is not in $\L^2(\R^{2n})$.
	\end{example}
	Now we consider a subspace of $L^2(\R^{2n})$ with the property that  any $F,G\in L^2(\R^{2n})$  such that $\widehat{F^k}\circledast \widehat{G^k}$ in $L^2(\R^{2n})$.
	
	For any $c\geqq 0, k\in {\R^m}^{\ast}$, we denote $W_c^{k}$ the set of all measurable functions $F$ on $\R^{2n}$ such that 
	$$\left|\widehat{F^{k}}(\xi,\eta)\right|\leq e^{-c|(B_{k}\xi,\eta)|^2}|f(\xi,\eta)|,\quad \xi,\eta\in \R^n,$$ for some function $f$ in $L^2(\R^{2n})$, where $F^{k}(q,p)=F\Big(B_{k}\dfrac{p}{\|k\|_2^2},-B_{k}^{t}\dfrac{q}{\|k\|_2^2}\Big)$. Note that for all $c\geq 0$, $W_c^{k}$ is a subspace of $L^2(\R^{2n})$ and $W_c^{k}\subseteq W_d^{k}$ if $c\geq d$.
	
	Next we show that the product of two localization operators with respect to symbols in $W_c^k$ is indeed a localization operator with respect to a symbol in $L^2(\R^{2n})$.
	\begin{theorem}
		Let $F$ and $G$ be in $W_c^{k}$, where $c>\frac{1+\sqrt{5}}{8}$. Then the product of localization operators $L^{k}_{F,\varphi}:L^2(\R^n)\to L^2(\R^n)$ and $L^{k}_{G,\varphi}:L^2(\R^n)\to L^2(\R^n)$ is a localization operator $L^{k}_{H,\varphi}:L^2(\R^n)\to L^2(\R^n)$, where $H\in \bigcap_{0<d\leq c_{\epsilon}}W_d^{k}$, and 
		\begin{equation}\label{conditions}
			c_{\epsilon}=c-c\epsilon-\frac{1}{4},\quad \frac{4c}{8c+1}<\epsilon<1-\frac{1}{4c}.
		\end{equation}
	\end{theorem}
	\begin{proof}
		Let $f$ and $g$ in $L^2(\R^{2n})$ such that 
		\begin{equation}\label{1}
			\left|\widehat{F^{k}}(\xi,\eta)\right|\leq e^{-c|(B_{k}\xi,\eta)|^2}|f(\xi,\eta)|,
		\end{equation}and 
		\begin{equation}\label{2}
			\left|\widehat{G^{k}}(\xi,\eta)\right|\leq e^{-c|(B_{k}\xi,\eta)|^2}|g(\xi,\eta)|,
		\end{equation} for all $\xi,\eta\in \R^n$.
		Now using \eqref{1},\eqref{2} and the definition of new convolution $\circledast$, we get
		\begin{align}\label{eq1}
			&\left|\Big(\widehat{F^{k}}\circledast\widehat{G^{k}}\Big)(\xi,\eta)\right|\nonumber\\
			& \leq \int_{\R^{2n}}|\widehat{F^{k}}(\xi- q,\eta-p)||\widehat{G^{k}}(q,p)| e^{-\frac{1}{2}\left(|p|^2+\|k\|_2^2|q|^2\right)}\times e^{\frac{1}{2}\left(\langle \eta,p\rangle+\|k\|_2^2\langle\xi,q\rangle\right)}dqdp\nonumber\\
			&\leq\int_{\R^{2n}}e^{-c\left|\left(B_{k}(\xi-q),(\eta-p)\right)\right|^2}|f(\xi-q,\eta-p)|e^{-c|(B_{k}q,p)|^2}|g(q,p)|\nonumber\\
			&\times e^{-\frac{1}{2}\left(|p|^2+\|k\|_2^2|q|^2\right)} e^{\frac{1}{2}\left(\langle \eta,p\rangle+\|k\|_2^2\langle\xi,q\rangle\right)}dqdp\nonumber\\
			&\leq\int_{\R^{2n}}e^{-c\left|\left(B_{k}(\xi-q),(\eta-p)\right)\right|^2}|f(\xi-q,\eta-p)|e^{-c|(B_{k}q,p)|^2}|g(q,p)|\nonumber\\
			&\times e^{-\frac{1}{2}\left(|p|^2+\|k\|_2^2|q|^2\right)} e^{\frac{1}{4}\left( |\eta|^2+|p|^2+\|k\|_2^2|\xi|^2+|q|^2\right)}dqdp.
		\end{align}
		Consider the powers of exponentials in  \eqref{eq1},
		\begin{align}\label{eq2}
			&-c\left(|B_{k}(\xi-q)|^2+|\eta-p|^2\right)-c\left(|B_{k}q|^2+|p|^2\right)+{\frac{1}{4}\left(|\eta|^2+|p|^2+\|k\|_2^2|\xi|^2+|q|^2\right)}\nonumber\\
			&-\frac{1}{2}\left(|p|^2+\|k\|_2^2|q|^2\right)\nonumber\\
			&=-c\Big(\|k\|_2^2|\xi|^2+\|k\|_2^2|q|^2-2\langle B_{k}\xi,B_{k}q\rangle+|\eta|^2+|p|^2-2\langle\eta,p\rangle\Big)\nonumber\\
			& -c\left(\|k\|_2^2|q|^2+|p|^2\right)+{\frac{1}{4}\left(|\eta|^2+|p|^2+\|k\|_2^2|\xi|^2+|q|^2\right)}-\frac{1}{2}\left(|p|^2+\|k\|_2^2|q|^2\right).
		\end{align}
		For each $\epsilon>0$, 
		\begin{align*}
			&	2c\langle B_{k}\xi,B_{k}q\rangle\\
			&\leq 2c|B_{k}\xi||B_{k}q|\\
			&=2c\sqrt{\epsilon}|B_{k}\xi|\dfrac{1}{\sqrt{\epsilon}}|B_{k}q|\leq c\|k\|_2^2\left(\epsilon|\xi|^2+\dfrac{1}{\epsilon}|q|^2\right),
		\end{align*}
		and  similarly
		\begin{align*}
			2c\langle\eta,p\rangle\leq c\left(\epsilon|\eta|^2+\dfrac{1}{\epsilon}|p|^2\right).
		\end{align*}
		For each $\epsilon>0$, \eqref{eq2} becomes 
		\begin{align*}
			&-c\left(|B_{k}(\xi-q)|^2+|\eta-p|^2\right)-c\left(|B_{k}q|^2+|p|^2\right)+{\frac{1}{4}\left(|\eta|^2+|p|^2+\|k\|_2^2|\xi|^2+|q|^2\right)}\\
			&-\frac{1}{2}\left(|p|^2+\|k\|_2^2|q|^2\right)\leq \left(\|k\|_2^2|\xi|^2+|\eta|^2\right)\left(-c+c\epsilon+\frac{1}{4}\right)+\left(|\eta|^2+|p|^2\right)\left(-2c+\frac{c}{\epsilon}-\frac{1}{4}\right).
		\end{align*}
		Now the inequality \eqref{eq1} can be written as
		\begin{align}\label{finalequation}
			&\left|\Big(\widehat{F^{k}}\circledast\widehat{G^{k}}\Big)(\xi,\eta)\right|\nonumber\\
			&\leq e^{-\left(c-c\epsilon-\frac{1}{4}\right)\left(\|k\|_2^2|\xi|^2+|\eta|^2\right)}\int_{\R^{2n}}|f(\xi-q,\eta-p)||g(q,p)|e^{-\left(2c-\frac{c}{\epsilon}+\frac{1}{4}\right)\left(|\eta|^2+|p|^2\right)}dqdp.
		\end{align}
		Hence from \eqref{conditions} and  \eqref{finalequation}, we get $H\in W_{c_{\epsilon}}^{k}\subset L^2(\R^{2n})$. Thus by the fact, $W_c^{k}\subseteq W_d^{k}$ if $c\geq d$, we have $$H\in \bigcap_{0<d\leq c_{\epsilon}}W_d^{k}.$$
	\end{proof}
	
	%\begin{theorem}
	%	If $F$ and $G$ are Gaussian functions then the product of localization operators $L_F:L^2(\R^n)\to L^2(\R^n)$ and $L_F:L^2(\R^n)\to L^2(\R^n)$ is a localization operator $L_H=L_FL_G$, where $H$ is also a Gaussian function. 
	%\end{theorem}
	%\begin{proof}
	%	Let $F(q,p)=e^{-c|(q,p)|^2}$ and $G(q,p)=e^{-d|(q,p)|^2}$
	%\end{proof}

	\section{Weyl Transform on Reduced Heisenberg group with Multi-dimensional center}\label{s4}
	In this section, we study the boundedness of Weyl transform associated to the Wigner transform on the reduced Heisenberg group with multidimensional center $\G$.
	
	Let us denote $L^r\left(\G\times\Z^{m\ast},S_r,(2\pi)^{-n}\|k\|^n_{2}d\mu(q,p,t)\right),$  $1\leq r <\infty$, the space of all operator valued  functions $F$ on $\G\times \Z^{m\ast}$ such that
	$$\|F\|_{r,\mu\otimes\alpha}^r=\int_{\G}\sum_{k\in\Z^{m\ast}}\|F(q,p,t,k)\|_{S_r}^r(2\pi)^{-n}\|k\|^n_{2} d\mu(q,p,t)<\infty,$$ and for $r=\infty$
	$$\|F\|_{\infty,\mu\otimes\alpha}=\operatorname{ess~sup}\limits_{(q,p,t,k)\in\G\times\Z^{m\ast}}\|F(q,p,t,k)\|_{S_{\infty}}<\infty,$$ where $\mu\otimes\alpha$ is the product measure on $\G\times\Z^{m\ast}$, given by  $$d(\mu\otimes\alpha)=d\mu(q,p,t) d\alpha(k)=d\mu(q,p,t)(2\pi)^{-n}\|k\|_2^n=\|k\|_2^n\dfrac{dqdpdt}{(2\pi)^{m+n}},$$ and $S_r$ denotes $r$-Schatten-von Neumann class.

	We want to show, the Weyl transform is a bounded operator on $L^2(\G)$ when the symbol belongs to $L^r(\G\times\Z^{m\ast},S_r,d(\mu\otimes\alpha)),$  $1\leq r \leq 2$, and for every $2<r<\infty$, there exists an unbounded Weyl transform with symbol in $L^r(\G\times\Z^{m\ast},S_r,d(\mu\otimes\alpha))$. 
	\begin{definition}\label{translation}
		The translation operator on $L^1(\G)$ is given by
		$$\tau_{(q^{\prime},p^{\prime},t^{\prime})}f(q,p,t)=f((q^{\prime},p^{\prime},t^{\prime})^{-1}\ast(q,p,t)),$$
		for all $(q^{\prime},p^{\prime},t^{\prime}),(q,p,t)\in \G$.
	\end{definition}
	\begin{definition}\label{Wigner}
		Let $f,g$ be in $S(\G)$. Then the Wigner transform associated to the reduced Heisenberg with multidimensional center $\G$ is defined by
		$$V(f,g)(q,p,t,k)=\int_{\G}f(q^{\prime},p^{\prime},t^{\prime})\tau_{(q^{\prime},p^{\prime},t^{\prime})}g(q,p,t)\pi_{k}(q^{\prime},p^{\prime},t^{\prime})d\mu(q^{\prime},p^{\prime},t^{\prime})\quad (q,p,t)\in \G, k\in\Z^{m\ast}.$$
	\end{definition}
	
	Now	we can write the Wigner transform as the group Fourier transform.
	$$V(f,g)(q,p,t,k)={(\F f\cdot\tau_{\cdot}g(q,p,t))}(k), (q,p,t)\in\G,k\in \Z^{m\ast}.$$
	
	\begin{proposition}\label{Moyal}
		Let $f_1,f_2,g_1,g_2$ be in $L^2(\G,d\mu)$. Then
		$$\langle V(f_1,g_1),V(f_2,g_2)\rangle_{L^2(\G\times\Z^{m\ast},S_2,d(\mu\otimes\alpha))}=\langle f_1-f_1^0,f_2-f_2^0\rangle_{L^2(\G,d\mu)}\langle g_1,g_2\rangle_{L^2(\G,d\mu)}.$$ 		
	\end{proposition}
	\begin{proof} Now by definitions \ref{Wigner}, \ref{translation} and Theorem \ref{Plancherel}, we get
		\begin{align*}
			&\langle V(f_1,g_1),V(f_2,g_2)\rangle_{L^2(\G\times\Z^{m\ast},S_2,d(\mu\otimes\alpha))}\\
			&= \sum_{k\in\Z^{m\ast}}\int_{\G}\text{tr}\left[V(f_1,g_1)(q,p,t,k)V(f_2,g_2)(q,p,t,k)^{\ast}\right](2\pi)^{-n}\|k\|_2^nd\mu(q,p,t)\\
			&=\int_{\G}\sum_{k\in\Z^{m\ast}}\text{tr}\left[\F(f_1\tau g_1(q,p,t))(k)\F(f_2\tau g_2(q,p,t))(k)^{\ast}\right](2\pi)^{-n}\|k\|_2^nd\mu(q,p,t)\\
			&= \langle f_1-f_1^0,f_2-f_2^0\rangle_{L^2(\G,d\mu)}\langle g_1,g_2\rangle_{L^2(\G,d\mu)}.
		\end{align*}
	\end{proof}
	\begin{proposition}\label{propo}
		Let $f$ be in $L^r(\G,d\mu)$ and $g$ be in $L^{r^{\prime}}(\G,d\mu)$, for $r\in [1,2]$. Then 
		$$\|V(f,g)\|_{r^{\prime},\mu\otimes\alpha}\leq c_r\|f\|_{L^r(\G)}\|g\|_{L^{r^{\prime}}(\G)},$$ where $c_r$ is a constant depending on $r$.
	\end{proposition}	
	\begin{proof}
		For $r=2$, by Proposition \ref{Moyal}, we get
		$$\|V(f,g)\|_{2,\mu\otimes\alpha}= \|f-f^0\|_{L^2(\G)}\|g\|_{L^2(\G)}\leq 2\|f\|_{L^2(\G)}\|g\|_{L^2(\G)}.$$
		For $r=1$,
		\begin{align}
			\|V(f,g)\|_{\infty,\mu\otimes\alpha}&=\operatorname{ess~sup}\limits_{(q,p,t,k)\in\G\times\Z^{m\ast}}\|V(f,g)(q,p,t,k)\|_{S_{\infty}},\nonumber\\
			&=\operatorname{ess~sup}\limits_{(q,p,t,k)\in\G\times\Z^{m\ast}}\|\F(f\tau g(q,p,t)(k))\|_{S_{\infty}}\nonumber\\
			&\leq c \operatorname{ess~sup}\limits_{(q,p,t)\in\G}\|f\tau g(q,p,t)\|_{L^1(\G)}\nonumber\\
			&=c\|f\|_{L^1(G)}\|g\|_{L^{\infty}(G)}.
		\end{align}
		Now the proof can be followed from interpolation theory of operators, \cite{interpolation}.
	\end{proof}	
	\begin{proposition}
		Let $f,g$ be in $L^2(\G)$. Then 
		$$V(f,g)\in L^r(\G\times\Z^{m\ast},S_r,d(\mu\otimes\alpha)), \quad r\in [2,\infty].$$
	\end{proposition}
	\begin{proof}
		The proof follows from Proposition \ref{propo} and using the fact $f,g\in L^2(G)$, then $fg\in L^1(G)$.
		
	\end{proof}
	The following theorem represents the group Fourier transform in terms of the Wigner transform.
	\begin{theorem}\label{FT}
		Let $f,g$ be in $L^1(\G,d\mu)\cap L^2(\G,d\mu)$ and 
		$C=\int_{\G}g(q,p,t)d\mu(q,p,t)\neq 0$. Then for any $k\in \Z^{m\ast}$, we have
		$$(\F f)(k)=\frac{1}{C}\int_{\G}V(f,g)(q,p,t,k)d\mu(q,p,t).$$
	\end{theorem}
	\begin{proof}
		Using the definition of translation operator and the group Fourier transform, we get 
		\begin{align*}
			&\dfrac{1}{C}\int_{\G}V(f,g)(q,p,t,k)d\mu(q,p,t)\\
			& \dfrac{1}{C}\int_{\G}\int_{\G}f(q^{\prime},p^{\prime},t^{\prime})g\left((q^{\prime},p^{\prime},t^{\prime})^{-1}(q,p,t)\right)\pi_k(q^{\prime},p^{\prime},t^{\prime})d\mu(q^{\prime},p^{\prime}t^{\prime})d\mu(q,p,t)\\
			&=\dfrac{1}{C}\int_{\G}\left(\int_{\G}f(q^{\prime},p^{\prime},t^{\prime})g(x,y,s)\pi_k(q^{\prime},p^{\prime},t^{\prime})d\mu(q^{\prime},p^{\prime},t^{\prime})\right)d\mu(x,y,s)\\
			&=\dfrac{1}{C}(\F f)(k)\int_{\G}g(x,y,s)d\mu(x,y,s)=(\F f)(k).
		\end{align*}
	\end{proof}
	The following corollary is the inversion formula for the Wigner transform.
	\begin{corollary}
		With the hypothesis of Theorem \ref{FT} and if ${\F f}(k)\in \ell^1(\Z^{m\ast},S_1,(2\pi)^{-n}\|k\|_2^n)$, then we have 
		$$f(q,p,t)-f^0(q,p,t)=C^{-1}\sum_{k\in\Z^{m\ast}}\operatorname{tr}\left[\pi_k^{\ast}(q,p,t)\int_{\G}V(f,g)(q^{\prime},p^{\prime},t^{\prime},k)d\mu(q^{\prime}p^{\prime}t^{\prime})\right](2\pi)^{-n}\|k\|_2^n.$$
	\end{corollary} 
	\begin{proof}
		By Fourier inversion theorem and Theorem \ref{FT}
		\begin{align*}
			f(q,p,t)-f^0(q,p,t)&=\sum_{k\in\Z^{m\ast}}\text{tr}\left[\pi_k^{\ast}(q,p,t)(\F f)(k)\right](2\pi)^{-n}\|k\|_2^n\\
			&=C^{-1}\sum_{k\in\Z^{m\ast}}\text{tr}\left[\pi_k^{\ast}(q,p,t)\int_{\G}V(f,g)(q^{\prime},p^{\prime},t^{\prime},k)d\mu(q^{\prime},p^{\prime},t^{\prime})\right](2\pi)^{-n}\|k\|_2^n.
		\end{align*}
	\end{proof}
	Now we study the Weyl transform associated to Wigner transform on the reduced Heisenberg group with multidimensional center.
	\begin{definition}
		Let $\sigma$ be in $L^p(\G\times\Z^{m\ast},S_p,d(\mu\otimes\alpha))$. Then the Weyl transform $W_{\sigma}$ of $\sigma$, an operator on some space of $G$, is defined by 
		\begin{align}
			\langle W_{\sigma}f,\overline{g}\rangle_{L^2(\G,d\mu)}&=\int_{\G}W_{\sigma}f(q,p,t)g(q,p,t)d\mu(q,p,t)\nonumber\\
			&=\langle V(f,g),\sigma\rangle_{\mu\otimes\alpha}\nonumber\\
			&=\int_{\G}\sum_{k\in\Z^{m\ast}}\text{tr}\left[\sigma^{\ast}(q,p,t,k)V(f,g)(q,p,t,k)\right](2\pi)^{-n}\|k\|_2^nd\mu(q,p,t),
		\end{align}
		for all $f,g$ in $S(G)$.
	\end{definition}
	The kernel representation of the Weyl transform $W_{\sigma}$ is given by
	\begin{equation}
		W_{\sigma}f(q,p,t)=\int_{\G}K(q,p,t,q^{\prime},p^{\prime},t^{\prime})f(q^{\prime},p^{\prime},t^{\prime})d\mu(q^{\prime},p^{\prime},t^{\prime}),
	\end{equation}
	with the kernel $$K(q,p,t,q^{\prime},p^{\prime},t^{\prime})=\sum_{k\in\Z^{m\ast}}\text{tr}\left[\sigma^{\ast}\left((q^{\prime},p^{\prime},t^{\prime})(q,p,t),k\right)\pi_k(q^{\prime},p^{\prime},t^{\prime})\right](2\pi)^{-n}\|k\|_2^n.$$ The following two theorems state the boundedness of Weyl transform when the symbol $\sigma$ belonging to $L^r(\G\times\Z^{m\ast},S_1,d(\mu\otimes\alpha)), 1\leq r\leq 2.$
	\begin{theorem}
		Let $\sigma$ be in  $L^1(\G\times\Z^{m\ast},S_1,d(\mu\otimes\alpha))$. Then the Weyl transform $W_{\sigma}$ on $L^2(\G)$ is in Schatten-von Neumann class $S_1$.
	\end{theorem}
	\begin{proof}
		Using the fact, $|\text{tr}(T)|\leq \text{tr}(|T|)$, we get
		\begin{align*}
			\|W_{\sigma}\|_{S_1}&=\int_{\G}|K(q,p,t,q,p,t)|d\mu(q,p,t)\\
			&= \int_{\G}\left|\sum_{k\in\Z^{m\ast}}\text{tr}\left[\sigma^{\ast}\left((q,p,t)(q,p,t),k\right)\pi_k(q,p,t)\right](2\pi)^{-n}\|k\|_2^n\right|d\mu(q,p,t)\\
			&=\int_{\G}\left|\sum_{k\in\Z^{m\ast}}\text{tr}\left[\sigma^{\ast}\left(q,p,t,k\right)\pi_k(q/2,p/2,t/2)\right](2\pi)^{-n}\|k\|_2^n\right|2^{-2n-m}d\mu(q,p,t)\\
			&\leq \int_{\G}\sum_{k\in\Z^{m\ast}}\text{tr}\left[|\sigma^{\ast}(q,p,t,k)|\right]2^{-2n-m}(2\pi)^{-n}\|k\|_2^nd\mu(q,p,t)\\
			&= 2^{-2n-m}\|\sigma\|_{1,\mu\otimes\alpha}.
		\end{align*}
	\end{proof}
	\begin{theorem}
		Let $\sigma$ be in $L^2(\G\times\Z^{m\ast},S_2,d(\mu\otimes\alpha))$. Then the Weyl transform $W_{\sigma}$ on $L^2(\G)$ is in the Schatten-von Neumann class $S_2$.
	\end{theorem}
	\begin{proof}
		By using the fact $\text{tr}(A^{\ast})=\text{tr}(A)$, the group Fourier inversion transform and Plancherel formula, we get
		\begin{align*}
			&\|W_{\sigma}\|^2_{S_2}\\
			&=\int_{\G}\int_{\G}|K(q,p,t,q^{\prime},p^{\prime},t^{\prime})|^2d\mu(q,p,t)d\mu(q^{\prime},p^{\prime},t^{\prime})\\
			&=\int_{\G}\int_{\G}\left|\sum_{k\in\Z^{m\ast}}\text{tr}\left[\sigma^{\ast}((q^{\prime},p^{\prime},t^{\prime})(q,p,t),k)\pi_k(q^{\prime},p^{\prime},t^{\prime})\right](2\pi)^{-n}\|k\|_2^n\right|^2d\mu(q,p,t)d\mu(q^{\prime},p^{\prime},t^{\prime})\\
			&=\int_{\G}\int_{\G}\left|\sum_{k\in\Z^{m\ast}}\text{tr}\left[\pi_k^{\ast}(q^{\prime},p^{\prime},t^{\prime})\sigma(q,p,t,k)\right](2\pi)^{-n}\|k\|_2^n\right|^2d\mu(q,p,t)d\mu(q^{\prime},p^{\prime},t^{\prime})\\
			&=\int_{\G}\int_{\G}\left|\left(\F^{-1}\sigma(q,p,t,\cdot)\right)(q^{\prime},p^{\prime},t^{\prime})-\left(\F^{-1}\sigma(q,p,t,\cdot)\right)^{0}(q^{\prime},p^{\prime},t^{\prime})\right|^2d\mu(q,p,t)d\mu(q^{\prime},p^{\prime},t^{\prime})\\
			&=\int_{\G}\sum_{k\in\Z^{m\ast}}\left\lVert\sigma(q,p,t,k)\right\rVert_{S_2}^2(2\pi)^{-n}\|k\|_2^nd\mu(q,p,t)\\
			&=\|\sigma\|_{2,\mu\otimes\alpha}.
		\end{align*}
	\end{proof}
	
	Here we  recall that Hermite polynomials $H_k(x)$ are of the form 
	$$H_k(x)=(-1)^k\dfrac{d^k}{dx^k}(e^{-x^2})e^{x^2},\quad k=0,1,2,\cdots$$ on the real line. Then  the Hermite functions $\widetilde{h_k}$ are defined  by
	$$\widetilde{h_k}(x)=H_k(x)e^{\frac{-1}{2}x^2}, k=0,1,2,\cdots,$$
	and  the normalized Hermite functions are given by
	$$h_k(x)=(2^k k!\sqrt{\pi})^{\frac{-1}{2}}\widetilde{h_k}(x).$$ Now we can define the Hermite functions on $\R^n$ by taking the product of $n$ Hermite functions on $\R$.
	Let $\gamma$ be a multi-index and $x\in\R^n$. The Hermite function $\Phi_{\gamma}$ is given by $$\Phi_{\gamma}(x)=\prod_{j=1}^{n}h_{\gamma_j}(x_j).$$ The set $\{\Phi_{\gamma}:\gamma\in\Z^n_{+}\}$ of all Hermite functions on $\R^n$ forms an orthonormal system for $L^2(\R^n)$. These functions are eigen vectors of the Hermite operator $H=\Delta+\|x\|^2_2$, i.e.,
	$$H\Phi_{\gamma}=(2|\gamma|+n)\Phi_{\gamma},$$ where $|\gamma|=\gamma_1+\gamma_2+\cdots+\gamma_n$ and $\|x\|^2=|x_1|^2+|x_2|^2+\cdots+|x_n|^2$.
	
	Let $\alpha>-1$. Then for $k=0,1,2,\cdots,$ the Laguerre polynomial of degree $k$ and order $\alpha$, $L_k^{\alpha}$ on $(0,\infty)$ is defined by 
	$$L_k^{\alpha}(x)=\frac{x^{-\alpha}e^x}{k!}\frac{d^k}{dx^k}(e^{-x}x^{\alpha+k}).$$
	We now recall the special Hermite functions as the Fourier-Wigner transform of the Hermite functions on $\R^n$. For each pair of multi-indices $\gamma$ and $\eta$, we define $\Phi_{\gamma\eta}(z)=\V(\Phi_{\gamma},\Phi_{\eta})(z)$. One can show that the special Hermite functions $\Phi_{\gamma\eta}$ forms a complete orthonormal system in $L^2(\R^{2n})$. Also there is a relation between special Hermite functions and Laguerre polynomials, which is,
	\begin{equation}\label{Lagurre}
		\Phi_{\gamma\gamma}(z)=(2\pi)^{-n/2}\prod_{j=1}^{n}L_{\gamma_j}^0(\frac{1}{2}|z_j|^2)e^{\frac{-1}{4}|z_j|^2}, \quad z_j=(x_j,y_j)\in\R^{2n}.
	\end{equation}
	In the next theorem we prove the unboundedness of the Weyl transform $W_{\sigma}$.
	\begin{theorem}\label{notbouned}
		For $r\in (2,\infty)$, there exists an symbol $\sigma $ in $L^r(\G\times\Z^{m\ast},S_r,d(\mu\otimes\alpha))$ such that $W_{\sigma}$ on $L^2(\G)$ is an unbounded operator.
	\end{theorem}
	To prove Theorem \ref{notbouned}, we follow the following lemmas.
	\begin{lemma}\label{uniformbound}
		Suppose that the Weyl transform, $W_{\sigma}$ is bounded linear operator on $L^2(\G)$ for all $\sigma$ in $L^r(\G\times\Z^m,S_r,d(\mu\otimes\alpha)), 2<r<\infty$. Then there exists a $c>0$ such that
		\begin{equation}\label{BoundedWeyl}
			\|W_{\sigma}\|\leq c\|\sigma\|_{r,\mu\otimes\alpha}.
		\end{equation}
	\end{lemma}
	\begin{proof}
		The proof follows from Banach-Steinhauss Theorem.
	\end{proof}
	\begin{lemma}\label{finite}
		Suppose the inequality \eqref{BoundedWeyl} holds for all  $r$, $2<r<\infty$. Then
		\begin{equation}\label{wigner}
			\|V(f,g)\|_{r^{\prime},\mu\otimes\alpha}\leq c\|f\|_{L^2(\G)}\|g\|_{L^2(G)},\quad \dfrac{1}{r}+\dfrac{1}{r^{\prime}}=1,
		\end{equation}
		for all $f,g\in L^2(\G)$.
	\end{lemma}
	\begin{proof}
		Let $f,g$ be in $L^2(\G)$. Then by the definition of the Weyl transform $W_{\sigma}$ and inequality \eqref{BoundedWeyl}, we get
		\begin{align*}
			\|V(f,g)\|_{r^{\prime},\mu\otimes\alpha}&=\sup_{\|\sigma\|_{r,\mu\otimes\alpha}=1}|\langle V(f,g),\sigma\rangle_{\mu\otimes\alpha}|\\
			&= \sup_{\|\sigma\|_{r,\mu\otimes\alpha}=1}|\langle W_{\sigma}f,\overline{g}\rangle_{L^2(\G)}|\\
			&\leq \sup_{\|\sigma\|_{r,\mu\otimes\alpha}=1}\|W_{\sigma}f\|_{L^2(\G)}\|g\|_{L^2(\G)}\\
			&\leq c \|f\|_{L^2(\G)}\|g\|_{L^2(\G)}.
		\end{align*}	
	\end{proof}
	\begin{lemma}\label{contradict}
		Suppose that the inequality \eqref{wigner} holds. Then
		$$\left(\sum_{k\in\Z^{m\ast}}\left\lVert\int_{\G}V(f,f)(q,p,t,k)d\mu(q,p,t)\right\rVert_{S_{r^{\prime}}}^{r^\prime}(2\pi)^{-n}\|k\|_2^n\right)^{1/r^{\prime}}<\infty,$$ for all compact support and integrable function $f$, on $\G$.
	\end{lemma}
	\begin{proof}
		Let $f$ be an integrable and compactly supported function in $\Omega\subset\G$. Then $V(f,f)$ is supported in $\overline{\Omega}\times \Z^{m\ast}$, where $\overline{\Omega}=\{(q,p,t)\ast(q^{\prime},p^{\prime},t^{\prime}):(q,p,t),(q^{\prime},p^{\prime},t^{\prime})\in \Omega\}$. Using Minkowski inequality for integral, H\"older inequality and \eqref{wigner}, we obtain
		\begin{align*}
			&\left(\sum_{k\in\Z^{m\ast}}\left\lVert\int_{\G}V(f,f)(q,p,t,k)d\mu(q,p,t)\right\rVert_{S_{r^{\prime}}}^{r^\prime}(2\pi)^{-n}\|k\|_2^n\right)^{1/r^{\prime}}\\
			&= \left(\sum_{k\in\Z^{m\ast}}\left(\sup_{\|\sigma\|_{r,\alpha}=1}\left|\langle\int_{\G}V(f,f)(q,p,t,\cdot)d\mu(q,p,t),\sigma(\cdot)\rangle\right|\right)^{r^\prime}(2\pi)^{-n}\|k\|_2^n\right)^{1/r^{\prime}}\\
			&= \left(\sum_{k\in\Z^{m\ast}}\left(\sup_{|\sigma\|_{r,\alpha}=1}\left|\text{tr}\left[\sigma(k)^{\ast}\int_{\G}V(f,f)(q,p,t,k)d\mu(q,p,t)\right]\right|\right)^{r^\prime}(2\pi)^{-n}\|k\|_2^n\right)^{1/r^{\prime}}\\
			&\leq \left(\sum_{k\in\Z^{m\ast}}\left(\int_{\G}\sup_{\|\sigma\|_{r,\mu\otimes\alpha}=1}\left|\text{tr}\left[\sigma(q,p,t,k)^{\ast}V(f,f)(q,p,t,k)\right]\right|d\mu(q,p,t)\right)^{r^\prime}(2\pi)^{-n}\|k\|_2^n\right)^{1/r^{\prime}}\\
			&=  \int_{\G}\left(\sum_{k\in\Z^{m\ast}}\left\lVert V(f,f)(q,p,t,k)\right\rVert_{S_{r^{\prime}}}^{r^\prime}(2\pi)^{-n}\|k\|_2^n\right)^{1/r^{\prime}}d\mu(q,p,t)\\
			&\leq \left(\int_{\Omega}d\mu(q,p,t)\right)^{1/r}\left(\int_{\Omega}\sum_{k\in\Z^{m\ast}}\left\lVert V(f,f)(q,p,t,k)\right\rVert_{S_{r^{\prime}}}^{r^\prime}(2\pi)^{-n}\|k\|_2^n d\mu(q,p,t)\right)^{1/r^{\prime}}<\infty.
		\end{align*}
		
	\end{proof}
	\begin{proof}[Proof of Theorem \ref{notbouned}]
		Let $f$ be an integrable function with compact support on $\G$ and $\int_{\G}f(q,p,t)d\mu(q,p,t)\neq 0$. Now 
		\begin{align}\label{final}
			&\dfrac{1}{|c|}\sum_{k\in\Z^{m\ast}}\left\lVert\int_{\G}V(f,f)(q,p,t,k)d\mu(q,p,t)\right\rVert_{S_{r^{\prime}}}^{r^{\prime}}(2\pi)^{-n}\|k\|_2^n\nonumber\\
			&=\sum_{k\in\Z^{m\ast}}\left\lVert(\F f)(k)\right\rVert_{S_{r^{\prime}}}^{r^{\prime}}(2\pi)^{-n}\|k\|_2^n\nonumber\\
			&\geq \sum_{k\in\Z^{m\ast}}\left\lVert(\F f)(k)\right\rVert_{S_{\infty}}^{r^{\prime}}(2\pi)^{-n}\|k\|_2^n\nonumber\\
			%		&\geq \sum_{k\in\Z^{m\ast}}\|(\F f)(k)\Phi_0^{k}\|_{L^2(\R^n)}^{r^{\prime}}(2\pi)^{-n}\|k\|_2^n\nonumber\\
			&\geq \sum_{k\in\Z^{m\ast}}\left|\langle(\F f)(k)\Phi_0^{k},\Phi_0^k\rangle_{L^2(\R^n)}\right|^{r^{\prime}}(2\pi)^{-n}\|k\|_2^n,
		\end{align}
		where $\Phi_{\gamma}^k(x)=\|k\|_2^{\frac{n}{4}}\Phi_{\gamma}(\sqrt{\|k\|_2}x)$ and $\Phi_{\gamma}$ are Hermite functions on $\R^n$. Now using the relation \eqref{Lagurre}, we get
		\begin{align}\label{Hermite}
			&\langle \pi_k(q,p,t)\Phi_0^k,\Phi_0^k\rangle_{L^2(\R^n)}\nonumber\\
			&=\int_{\R^{n}}e^{ik\cdot t}e^{iq\cdot B_{k}(x+p/2)}\Phi_0^k(x+p)\overline{\Phi_0^{k}(x)}dx\nonumber\\
			&=e^{ik\cdot t}\int_{\R^{n}}e^{iq\cdot B_ky}\Phi_0^k(y+p/2)\overline{\Phi_0^k(y-p/2)}dy\nonumber\\
			&= \|k\|_2^{\frac{n}{2}}e^{ik\cdot t}\int_{\R^{n}}e^{iq\cdot B_ky}\Phi_0\left(\sqrt{\|k\|_2}(y+p/2)\right)\overline{\Phi_0\left(\sqrt{\|k\|_2}(y-p/2)\right)}dy\nonumber\\
			&= e^{ik\cdot t}\int_{\R^{n}}e^{iy\cdot (\frac{B_k^tq}{\sqrt{\|k\|_2}})}\Phi_0\left(y+\frac{\sqrt{\|k\|_2}p}{2}\right)\overline{\Phi_0\left(y-\frac{\sqrt{\|k\|_2}p}{2}\right)}dy\nonumber\\
			&=e^{ik\cdot t}V(\Phi_0,\Phi_0)\left(\frac{B_k^tq}{\sqrt{\|k\|_2}},\sqrt{\|k\|_2}p\right)\nonumber\\
			&=e^{ik\cdot t}\Phi_{00}\left(\frac{B_k^tq}{\sqrt{\|k\|_2}},\sqrt{\|k\|_2}p\right)\nonumber\\
			&=(2\pi)^{-n/2}e^{ik\cdot t}\prod_{j=1}^nL_{0}^0\left(\frac{|z_j|^2}{2}\right)e^{-\frac{|z_j|^2}{4}}\nonumber\\
			&=(2\pi)^{-n/2}e^{ik\cdot t}\prod_{j=1}^nL_{0}^0\left(\frac{|z_j|^2}{2}\right)e^{-\frac{\|k\|_2}{4}|(q,p)|^2}\nonumber\\
			&= (2\pi)^{-n/2}e^{ik\cdot t}e^{-\frac{\|k\|_2}{4}|(q,p)|^2},
		\end{align} 
		where $z_j=\left(\frac{B_{k}^tq}{\sqrt{\|k\|_2}},\sqrt{\|k\|_2}p\right)_j, j=1,2,\cdots,n$ and $L_{0}^0\left(\frac{|z_j|^2}{2}\right)=1$.

		Suppose
		\begin{equation}\label{example}
			f_{\alpha} = \begin{cases}
				\prod_{\iota=1}^m|t_{\iota}|^{\alpha}\prod_{j=1}^n|q_j|^{\alpha}\prod_{l=1}^n|p_l|^{\alpha}, &\text{ when $(q,p,t)\in A$ and $q_j,p_l,t_{\iota}\neq 0$},\\
				0 &\text{otherwise},
			\end{cases}
		\end{equation}
		
		where $A=\{(q,p,t):t_{\iota}\in[0,2\pi],|q_j|\leq 1,|p_l|\leq 1\}$ and $\frac{-1}{2}<\alpha\leq \frac{1}{r^{\prime}}-1$. Then $f_{\alpha}\in L^2(\G)$. From \eqref{Hermite}, we get
		\begin{align}\label{w}
			&	\langle (\F f_{\alpha})(k)\Phi_0^k,\Phi_0^k \rangle_{L^2(\R^n)}\nonumber\\
			&=(2\pi)^{-n/2}\int_{\T^m}\int_{\R^{n}}\int_{\R^{n}}\prod_{\iota=1}^m|t_{\iota}|^{\alpha}\prod_{j=1}^n|q_j|^{\alpha}\prod_{l=1}^n|p_l|^{\alpha}e^{ik\cdot t}e^{-\frac{\|k\|_2}{4}|(q,p)|^2}d\mu(q,p,t)\nonumber\\
			&=(2\pi)^{-n/2}\prod_{\iota=1}^m\left(\int_0^{2\pi}t_{\iota}^{\alpha}e^{ik_{\iota}t_{\iota}}\frac{dt_{\iota}}{2\pi}\right)\prod_{j=1}^n\left(\int_{-1}^1|q_j|^{\alpha}e^{-\frac{\|k\|_2}{4}q_j^2}dq_j\right)\nonumber\\
			&\quad\quad\quad\times\prod_{l=1}^n\left(\int_{-1}^1|p_l|^{\alpha}e^{-\frac{\|k\|_2}{4}p_l^2}dp_l\right).
		\end{align}
		Consider, for $k_{\iota}>0,{\iota}=1,2,\cdots,m,$
		$$\int_0^{2\pi}t_{\iota}^{\alpha}e^{ik_{\iota}t_{\iota}}dt_{\iota}=\dfrac{1}{k_{\iota}^{\alpha+1}}\int_0^{2\pi k_{\iota}}t_{\iota}^{\alpha}\cos t_j \frac{dt_{\iota}}{2\pi},$$ and
		\begin{align*}
			\int_{-1}^1|q_j|^{\alpha}e^{-\frac{\|k\|_2}{4}q_j^2}dq_j&=2\int_{0}^1q_j^{\alpha}e^{-\frac{\|k\|_2}{4}q_j^2}dq_j\\
			&=\left(\dfrac{2}{\sqrt{\|k\|_2}}\right)^{\alpha+1}\int_{0}^{\frac{\|k\|_2}{4}}q_j^{\frac{\alpha-1}{2}}e^{-q_j}dq_j.
		\end{align*}
		Similarly, $$\int_{-1}^1|p_l|^{\alpha}e^{-\frac{\|k\|_2}{4}p_l^2}dp_l=\left(\dfrac{2}{\sqrt{\|k\|_2}}\right)^{\alpha+1}\int_{0}^{\frac{\|k\|_2}{4}}p_l^{\frac{\alpha-1}{2}}e^{-p_l}dp_l.$$
		Hence for $k_{\iota}>0,{\iota}=1,2,\cdots,m$, the equation \eqref{w} becomes
		\begin{align}\label{final2}
			&\langle (\F f_{\alpha})(k)\Phi_0^k,\Phi_0^k \rangle_{L^2(\R^n)}\nonumber\\
			&= (2\pi)^{-n/2}\prod_{{\iota}=1}^m\dfrac{1}{k_{\iota}^{\alpha+1}}\int_0^{2\pi k_{\iota}}t_{\iota}^{\alpha}\cos t_{\iota} \frac{dt_{\iota}}{2\pi}\times\left(\dfrac{2}{\sqrt{\|k\|_2}}\right)^{2n(\alpha+1)}\left(\int_{0}^{\frac{\|k\|_2}{4}}x^{\frac{\alpha-1}{2}}e^{-x}dx\right)^{2n}\nonumber\\
			&\geq (2\pi)^{-n/2}\dfrac{1}{\|k\|_2^{m(\alpha+1)}}\prod_{{\iota}=1}^m\int_0^{2\pi k_{\iota}}t_{\iota}^{\alpha}\cos t_{\iota} \frac{dt_{\iota}}{2\pi}\times\left(\dfrac{2}{\sqrt{\|k\|_2}}\right)^{2n(\alpha+1)}\left(\int_{0}^{\frac{\|k\|_2}{4}}x^{\frac{\alpha-1}{2}}e^{-x}dx\right)^{2n}\nonumber\\
			&=(2\pi)^{-n/2}\dfrac{2^{2n(\alpha+1)}}{\|k\|_2^{(m+n)(\alpha+1)}}\prod_{{\iota}=1}^m\int_0^{2\pi k_{\iota}}t_{\iota}^{\alpha}\cos t_{\iota} \frac{dt_{\iota}}{2\pi}\times\left(\int_{0}^{\frac{\|k\|_2}{4}}x^{\frac{\alpha-1}{2}}e^{-x}dx\right)^{2n}.
		\end{align}
		Since $$\int_0^{\infty}t_{\iota}^{\alpha}\cos t_{\iota} \frac{dt_{\iota}}{2\pi}=\Gamma(\alpha)\cos\left(\frac{\pi}{2}\alpha\right)$$ and
		$$\int_0^{\infty}x^{\frac{\alpha-1}{2}}e^{-x}dx=\Gamma\left(\dfrac{\alpha+1}{2}\right).$$ So there exist natural numbers $\lambda_1,\lambda_2,\cdots,\lambda_m,a$ and positive number $M$ such that 
		$$\left|\int_0^{2\pi k_{\iota}}t_{\iota}^{\alpha}\cos t_{\iota} \frac{dt_{\iota}}{2\pi}\right|\geq M,\quad \text{for}\quad k_{\iota}>\lambda_{\iota}, {\iota}=1,2,\cdots,m,$$ and
		$$\int_{0}^{\frac{\|k\|_2}{4}}x^{\frac{\alpha-1}{2}}e^{-x}dx\geq M,\quad \text{for}\quad \|k\|_2>a.$$ Choose $\lambda_0=\max\{\lambda_1,\lambda_2,\cdots,\lambda_m,a\}.$ Then
		$$\left|\int_0^{2\pi k_{\iota}}t_{\iota}^{\alpha}\cos t_{\iota} \frac{dt_{\iota}}{2\pi}\right|\geq M,\int_{0}^{\frac{\|k\|_2}{4}}x^{\frac{\alpha-1}{2}}e^{-x}dx\geq M,\quad\text{for}\quad k_{\iota}>\lambda_0, {\iota}=1,2,\cdots,m.$$
		Thus the inequality \eqref{final2} becomes 
		\begin{equation}\label{inequality}
			\left|\langle (\F f_{\alpha})(k)\Phi_0^k,\Phi_0^k \rangle_{L^2(\R^n)}\right|\geq \dfrac{2^{2n(\alpha+1)}}{(2\pi)^{n/2}}\times\dfrac{M^{2n+m}}{\|k\|_2^{(m+n)(\alpha+1)}},\quad \text{for}\quad k_{\iota}>\lambda_0,{\iota}=1,2,\cdots,m.
		\end{equation}
		Now using the inequality \eqref{inequality} in  \eqref{final}, we get
		\begin{align}\label{divergence}
			&\dfrac{1}{|c|}\sum_{k\in\Z^{m\ast}}\left\lVert\int_{\G}V(f_{\alpha},f_{\alpha})(q,p,t,k)d\mu(q,p,t)\right\rVert_{S_{r^{\prime}}}^{r^{\prime}}(2\pi)^{-n}\|k\|_2^n\nonumber\\
			&\geq \sum_{k\in\Z^{m\ast}}\left|\langle(\F f_{\alpha})(k)\Phi_0^{k},\Phi_0^k\rangle_{L^2(\R^n)}\right|^{r^{\prime}}(2\pi)^{-n}\|k\|_2^n\nonumber\\
			&\geq \dfrac{2^{2n(\alpha+1)}}{(2\pi)^{n/2}}\sum_{k_1>\lambda_1}\cdots\sum_{k_m>\lambda_0}\dfrac{M^{(2n+m)r^{\prime}}}{\|(k_1,\cdots,k_m)\|_2^{[(m+n)(\alpha+1)]r^{\prime}}}(2\pi)^{-n}\|(k_1,\cdots,k_m)\|_2^n=\infty, 
		\end{align} when $(m+n)(\alpha+1)r^{\prime}-n\leq m$, i.e., for $\alpha\leq \frac{1}{m+n}\left(\frac{n+m}{r^{\prime}}\right)-1$. The divergence of the inequality \eqref{divergence} is because of Remark \ref{ptest}. Hence for each $r$, $2<r<\infty$, we have an integrable, compactly supported function $f_{\alpha}, \frac{-1}{2}<\alpha\leq\frac{1}{r^{\prime}}-1$ such that 
		$$\sum_{k\in\Z^{m\ast}}\left\lVert\int_{\G}V(f_{\alpha},f_{\alpha})(q,p,t,k)d\mu(q,p,t)\right\rVert_{S_{r^{\prime}}}^{r^\prime}(2\pi)^{-n}\|k\|_2^n=\infty,$$
		which contradict Lemma \ref{contradict}. It completes the proof.
	\end{proof}
	\begin{remark}\label{ptest}
		Let $f(x_1,x_2,\cdots,x_m)=\dfrac{1}{(x_1^2+x_2^2+\cdots+x_m^2)^{p/2}}$, for all $x_i\geq 1, i=1,2,\cdots,m.$ Let $k_1,k_2,\cdots,k_m$ be positive integers. Then 
		\begin{equation}\label{inequality1}
			f(k_1,k_2,\cdots,k_m)\geq f(x_1,x_2,\cdots,x_m),
		\end{equation}
		when $x_i$ are in the closed intervals $ [k_i,k_i+1],i=1,2,\cdots,m$.
		Now taking integration on both side of the inequality \eqref{inequality1} over an rectangle $[k_1,k_1+1]\times [k_2,k_2+1]\times\cdots\times[k_m,k_m+1]$, we obtain
		$$f(k_1,k_2,\cdots,k_m)\geq \int_{k_1}^{k_1+1}\int_{k_2}^{k_2+1}\cdots\int_{k_m}^{k_m+1}f(x_1,x_2,\cdots,x_m)dx_1dx_2\cdots dx_m$$
		Summing the above inequality from $1$ to $\infty$, we get
		\begin{equation}\label{integral}
			\sum_{k_1=1}^{\infty}\sum_{k_2=1}^{\infty}\cdots\sum_{k_m=1}^{\infty} f(k_1,k_2,\cdots,k_m)\geq \int_{1}^{\infty}\int_{1}^{\infty}\cdots\int_{1}^{\infty}f(x_1,x_2,\cdots,x_m)dx_1dx_2\cdots dx_m.
		\end{equation} The right hand side of inequality \eqref{integral} diverges when $p\leq m$. Hence the left hand side series of \eqref{integral} diverges for $p\leq m.$
	\end{remark}
	\begin{obs}
		Here we defined the localization operator with respect to square integrable representation on the reduced Heisenberg group with multi-dimensional center $\G$. And derived the product formula of localization operators using a new convolution defined in \eqref{newconvolution}. But in the case of non-isotropic Heisenberg group with multi-dimensional center $\H^m$ \cite{shahla}, it is impossible to define localization operator, as all the irreducible representations are no more square integrable.
	\end{obs}
	\begin{obs} If we consider here the non-isotropic Heisenberg group with multi-dimensional center $\H^m$. Then all results about Weyl transform will be coming out as method we followed for the reduced Heisenberg group with multi-dimensional center $\G$. To prove the unboundedness of the Weyl transform on $L^2(\H^m)$ (Theorem \ref{notbouned}), the function defined in
		\eqref{example} will also work in this case,
		for each $r,2<r<\infty$, i.e., there exists a symbol $\sigma$ in $L^r(\H^m\times \R^{m\ast},S_r)$ such that the Weyl transform is unbounded operator on $L^2(\H^m)$.
	\end{obs}


\begin{thebibliography}{amsplain}
		
		
		\normalsize
		\baselineskip=17pt
		
		\bibitem{interpolation} Bennett, C. and Sharpley R. C.  \emph{Interpolation of operators}. Academic press(1988).
		
		
		\bibitem{Chen} Chen, L. and Zhao, J. ``Weyl transform and generalized spectrogram associated with quaternion Heisenberg group." \emph{Bulletin des Sciences Mathématiques}, \emph{136}(2), 127--143(2012).
		
		
		\bibitem{Cor}  Cordero, E. and  Gr\"ochenig, K. ``On the product of localization operators." \emph{Operator Theory:
			Advances and Applications}, \emph{172}, 279--295(2006).
		
		\bibitem{Dau1} Daubechies, I. ``{Time-frequency localization operators: a geometric phase space approach.}" \emph{IEEE Trans. Inform. Theory} \emph{34}(4), 605--612(1988).
		
		\bibitem{Dau2} Daubechies, I. \emph{Ten Lectures on Wavelets}." {SIAM}, (1992).
		
		\bibitem{Mari}De Mari, F., Feichtinger H. G. and  Nowak K. ``{Uniform eigenvalue estimates for time-frequency localization operators.}" \emph{Journal of the London Mathematical Society}, \emph{65}(3), 720--732(2002).
		
		
		\bibitem{Du} Du, J. and Wong M. W. ``{A product formula for localization operators}." \emph{Bulletin of Korean mathematical Society}, \emph{37}(1), 77--84(2000).
		
		
		\bibitem{Kohn} Kohn, J. J. and Nirenberg, L. ``An algebra of pseudo-differential operators." \emph{Communications on Pure and Applied mathematics}, \emph{18}(1), 269--305(1965).
		
		\bibitem{Liu} Liu, Y and  Wong, M. W. ``Polar wavelet transforms and localization operators." \emph{Integral Equations and Operator Theory}, \emph{58}, 99--110(2007).
		
		\bibitem{shahla} Molahajloo, S. ``Pseudo-differential operators on non-isotropic Heisenberg groups with multidimensional centers." \emph{Pseudo-Differential Operators: Groups, Geometry and Applications},
		15--35. Springer, 2017.
		
		\bibitem{Peng} Peng, L. and Zhao, J. `` Weyl transforms associated with the Heisenberg group." \emph{Bulletin des Sciences Mathématiques}, \emph{132}, 78--86(2008).
		
		\bibitem{Peng1} Peng, L. and Zhao, J. ``Weyl transforms on the upper half plane." \emph{Revista Matemática Complutense}, \emph{23}, 77--95(2010).
		
		
		\bibitem{Rachdi} Rachdi, L. T. and Trim\'eche, K. ``Weyl transforms associated with the spherical mean operator." \emph{Analysis and Applications}, \emph{1}(2), 141--164(2003).
		
		\bibitem{Simon} Simon, B. `` The Weyl transforms and $L^p$ functions on phase space." \emph{Proceedings of the American Mathematical Society}, \emph{116}(4), 1045--1047(1992).
		
		\bibitem{stien} Stien, E. M. and Murphy, T. S. \emph{Harmonic Analysis: Real-Variable Methods, Orthogonality, and Oscillatory Integrals}. Princeton University Press, New Jersey(1993).
		
		\bibitem{thang} Thangavelu, S. \emph{Harmonic analysis on the Heisenberg group}. volume 159. Springer Science and Business Media(1998).
		
		\bibitem{Weyl} Weyl, H. \emph{The theory of groups and quantum mechanics}. Dover(1950).
		
		\bibitem{wong1} Wong, M. W. \emph{The Weyl Transform}. Springer(1998).
		
		\bibitem{wong2} Wong, M. W. ``Localization operators on the affine group and paracommutators." \emph{Progress in Analysis}, 663--669(2003).
		
		\bibitem{Zhao} Zhao, J. M. and Peng, L. Z. ``Wavelet and Weyl transform associated with the spherical mean operator." \emph{Integral Equation and Operator Theory}, \emph{50}, 279--290(2004).		
		
	\end{thebibliography}
\end{document}